\documentclass[11pt,leqno]{amsart}
\usepackage[letterpaper,margin=1.5in,footskip=0.25in]{geometry}

\usepackage[T1]{fontenc}
\usepackage{tgschola}

\usepackage{mathrsfs}
\usepackage{microtype}
\usepackage{mathtools}
\usepackage{enumitem}

\PassOptionsToPackage{dvipsnames}{xcolor}
\usepackage{xcolor}

\usepackage{tikz-cd}

\PassOptionsToPackage{pdfusetitle,pagebackref,colorlinks}{hyperref}
\usepackage{hyperref}
\usepackage{bookmark}
\hypersetup{citecolor=Green,linkcolor=red,urlcolor=blue}
\usepackage[hyphenbreaks]{breakurl}

\usepackage{amssymb}
\usepackage{amsfonts}
\usepackage{bbm}
\usepackage{upgreek}
\usepackage{bm}
\usepackage{stmaryrd}
\usepackage{comment}
\usepackage{graphicx}
\usepackage{todonotes}

\allowdisplaybreaks

\makeatletter
\providecommand{\leftsquigarrow}{%
  \mathrel{\mathpalette\reflect@squig\relax}%
}
\newcommand{\reflect@squig}[2]{%
  \reflectbox{$\m@th#1\rightsquigarrow$}%
}
\makeatother

\newcommand{\A}{\mathbb{A}}

\newcommand{\C}{\mathbb{C}}

\newcommand{\Q}{\mathbb{Q}}
\newcommand{\R}{\mathbb{R}}
\newcommand{\Z}{\mathbb{Z}}

\newcommand{\F}{\mathbb{F}}

\newcommand{\cJ}{\mathcal{J}}

\newcommand{\cO}{\mathcal{O}}

\newcommand{\fa}{\mathfrak{a}}

\newcommand{\m}{\mathfrak{m}}

\newcommand{\End}{\operatorname{End}}

\DeclareMathOperator{\Exc}{Exc}

\DeclareMathOperator{\Spec}{Spec}
\DeclareMathOperator{\vol}{vol}

\DeclareMathOperator{\ord}{ord}
\DeclareMathOperator{\Nef}{Nef}

\DeclareMathOperator{\Supp}{Supp}

\DeclareMathOperator{\rk}{rk}
\DeclareMathOperator{\Tr}{Tr}

\newcommand{\ceil}[1]{\left\lceil #1 \right\rceil}
\newcommand{\floor}[1]{\left\lfloor #1 \right\rfloor}
\newcommand{\NS}{\operatorname{NS}}
\newcommand{\nef}{\operatorname{Nef}}
\newcommand{\psef}{\operatorname{\overline{Eff}}}

\newcommand{\taub}{\tau_b}
\newcommand{\Frob}{F}
\newcommand{\Frobpush}[1]{\Frob^{#1}_*}

\newcommand{\w}{\omega}

\newcommand{\Hom}{\mathrm{Hom}}

\newcommand{\Tot}{\operatorname{Tot}}

\newcommand{\Eff}{\operatorname{Eff}}

\newcommand{\codim}{\operatorname{codim}}

\numberwithin{equation}{section}

\newtheorem{prop}{Proposition}[section]
\newtheorem{thm}[prop]{Theorem}

\newtheorem{lem}[prop]{Lemma}
\newtheorem{cor}[prop]{Corollary}
\newtheorem{prop-def}[prop]{Proposition-Definition}

\newtheorem{theorem}[prop]{Theorem}
\newtheorem{lemma}[prop]{Lemma}
\newtheorem{corollary}[prop]{Corollary}

\newtheorem{proposition}[prop]{Proposition}
\newtheorem{thm-defn}[prop]{Theorem-Definition}
\newtheorem{assume}[prop]{Assumption}

\theoremstyle{definition}

\newtheorem{que}[prop]{Question}

\newtheorem{exam}[prop]{Example}

\newtheorem{setup}[prop]{Setup}

\theoremstyle{remark}

\newtheorem{remark}[prop]{Remark}

\newtheorem{definition}[prop]{Definition}

\newtheorem*{ThA*}{\textbf{Theorem 1}}
\newtheorem*{ThB*}{\textbf{Theorem 2}}
\newtheorem*{ThC*}{\textbf{Theorem C}}
\newtheorem*{ThD*}{\textbf{Theorem D}}
\newtheorem*{ThE*}{\textbf{Theorem E}}
\newtheorem*{Con*}{Conjecture}

\title{$\mathbf{F}$-jumping numbers can be irrational}
\author{\fontfamily{cmtt} Rahul Ajit}
\address{Department of Mathematics\\ The University of Utah\\ Salt Lake City, UT 84112, USA.}
\email{rahulajit@math.utah.edu}
\date{\today}

\begin{document}
\begin{abstract}
Let $k$ be an $F$--finite and infinite field of characteristic $p>2$.  We show, there exist infinitely many $F$--finite local domains $(R,\m)$ which are not $\Q$--Gorenstein and $\tau_{\mathrm{b}}(R;\m^t)$ has all but finitely many \emph{irrational} $F$--jumping numbers. 
\end{abstract}

\maketitle

\tableofcontents

\newpage

\section{Introduction}

Test ideals originated in tight closure theory of Hochster \& Huneke in characteristic \(p>0\), and play crucial role in studying $F$-singularities, analogous of multiplier ideals, see \cite{HH,ST-survey, FBook} for a comprehensive introduction.  In the presence of a boundary, one studies triples \((X,\Delta;\fa)\), where \(X\) is a normal \(F\)-finite scheme essentially of finite type over a field of characteristic \(p>0\), \(\Delta\) is an effective \(\Q\)-divisor such that \(K_X+\Delta\) is \(\Q\)-Cartier, and \(\fa\subseteq \cO_X\) is an ideal sheaf.  For \(t\in\R_{\ge 0}\) one defines test ideals \(\tau(X,\Delta;\fa^t)\subseteq \cO_X\) for triples (see \cite[Definition 3.4]{ST}), and the values of \(t\) at which \(\tau(X,\Delta;\fa^t)\) ``jumps" are called the \(F\)-jumping numbers, see \cite{TakagiWatanabe,Mus-TW, FBook} for details.

A striking feature of the log \(\Q\)-Gorenstein setting is that the jumping behavior of \(\tau(X,\Delta;\fa^t)\) is arithmetically rigid: the \(F\)-jumping numbers of \(\tau(X,\Delta;\fa^t)\) form a discrete set of rational numbers by the works of many mathematicians including Blickle, Lyubeznik, Musta{\c{t}}{\u{a}}, Smith, Schwede, Takagi, Tucker, Zhang, etc; see \cite{BMS-1,BMS-2,BST-alt,BSTZ10,Fin-gen, S-Note-Fjump, KLZ,S-Graf, ST, STW-alt,Ajit-Simper-1} and the references therein.

The situation changes dramatically for \emph{non-\(\Q\)-Gorenstein} varieties. De Fernex--Hacon, in \cite{dFH} introduced multiplier ideals on normal (possibly non-\(\Q\)-Gorenstein) varieties and showed that they can be computed by a \emph{single} boundary, see \cite[\S 5]{dFH}. Using this theory, Urbinati then constructed an explicit example whose multiplier ideal have irrational jumping numbers \cite{Urbinati}. 

Even though basepoint free Bertini fails in positive charactersitic, Schwede, in \cite{SchwedeBig}, proved the following formula analogous to \cite{dFH}:
\[
\tau_b(X;\fa^t)=\sum_{\Delta}\tau(X,\Delta;\fa^t),
\]
where the sum runs over effective \(\Q\)-divisors \(\Delta\) for which \(K_X+\Delta\) is \(\Q\)-Cartier with index not divisible by \(p\). Note that even though each individual summand has only rational/discrete jumping behavior, the sum \(\tau_b\) may not. This sum formula will play a crucial role in this note.  

In this context, motivated by characteristic 0 example \cite{Urbinati}, it has long been suspected by experts that one may encounter irrational $F$-jumps as well, see \cite[Question 6.2]{BSTZ10}, \cite{S-Note-Fjump}, \cite[Section 7]{PST-SRI}, \cite[pg 474]{S-Graf} and \cite[Chapter 9, Question 3.5]{FBook} for example. The goal of this paper is to confirm this suspicion by showing infinitely many examples of cones over abelian and K3 surfaces having \emph{all but finitely many} irrational $F$-jumping numbers. We note that most of these examples are new, even in char. 0 and all of them have \emph{all} jumping numbers irrational.

\medskip

\noindent\textbf{Main result and strategy.} 

The main tool to compute $F$-jumps in these examples is the following theorem that follows directly from the works of Schwede and Tucker in \cite[Section 5]{ST}:

\begin{ThA*} (cf \cite[Theorem 5.1]{ST}, Theorem \ref{thm:asymptotic-real})
Let \(X=\Spec R\) be normal, affine, integral, and
\(F\)-finite of characteristic \(p>0\).  Let \(\Delta\ge 0\) be a
\(\Q\)-divisor such that \(K_X+\Delta\) is \(\Q\)-Cartier, and let
\(\mathfrak a\subset R\) be a nonzero ideal.  Assume that the normalized
blowup \(\pi:Y\to X\) of \(\mathfrak a\) is regular, that
\(\mathfrak a\cO_Y=\cO_Y(-G)\) for an effective Cartier divisor \(G\), that
\(-G\) is \(\pi\)-ample, and that
\(\Supp(G\cup \Exc(\pi)\cup \pi_*^{-1}\Delta)\) is simple normal crossings. Then, there exist integer $T_0\ge 0$ such that
\[
\tau(X,\Delta;\mathfrak{a}^t)=\cJ(X,\Delta;\mathfrak{a}^t)
\qquad\text{for all real }t\ge T_0.
\]
\end{ThA*}
One should think this as a generalization of the fact that the test ideal of an SNC pair is the multiplier ideal. We use this to prove:

\begin{ThB*}\label{thm:main}
There exists infinitely many $F$-finite, local domain $(R,\m)$ of characteristic $p>2$ and dimension $\ge 3$
such that $R$ is not $\mathbb{Q}$-Gorenstein and $\tau_b(R;\m^t)$ has all but finitely many irrational $F$-jumping numbers.
More precisely, we show in \S \ref{high-dim-example}, that given any $g \ge 2$, there exist an algebraic irrational number $\kappa\in(0,1)$ of degree $g$ such that
\[
n+\kappa \quad\text{is an $F$-jumping number of } \taub(R;\m^t) \quad\text{for all integers } n\ge \# \ \mathrm{generators \ of \ } \m
\]
where dim$(R) = g+1$.
\end{ThB*}

The number $\kappa$ arises from an irrational \emph{pseudo-effective threshold} on a smooth polarized projective variety $(W,A)$:
\[
t_0(W,A) := \inf\{t\in\R \mid tA - K_W \ \text{is pseudo-effective}\},
\qquad
\kappa := 1 - t_0(W,A).
\]
Here $R$ is the local ring at the vertex of cone over $(W,A)$.

\smallskip

We compute \(\cJ\bigl(C,\Gamma_s;\m^t\bigr)\) explicitly for suitable boundaries $\Gamma_s$ on \(\pi\), and then apply \textbf{Theorem 1} to identify
\(\tau\bigl(R,\Gamma_s;\m^t\bigr)\) with the same explicit formula for all \(t\gg 0\).
Finally, Schwede's sum \cite{SchwedeBig} and Skoda's Theorem convert these explicit formulas into
irrational \(F\)-jumping behavior for \(\tau_b(R;\m^t)\).

We point out that these examples covers almost all positive algebraic irrational numbers. We end this short note by asking two questions.

\medskip

\noindent\textbf{Acknowledgments.}
I am extremely grateful to my advisors Christopher Hacon \& Karl Schwede for suggesting this problem, their constant encouragements, unwavering support, inspiring teachings, and infinite patience. Warm thanks are due to Yu-Ting Huang for a lot of stimulating discussions surrounding \cite{dFH}, 3 years ago, when we were learning these questions together and Christopher Hacon for suggesting to work out more examples and extremely helpful conversations on \S.4. I am indebted to Karl Schwede for his numerous insightful comments (including Theorem \ref{thm:asymptotic-general} and Theorem \ref{thm:dFH-direct}) and many expository suggestions which improved this note. Finally, it is my pleasure to thank Gari Chua for helpful feedback on a previous draft, Yotam Svoray for correcting a few TeX typos and Daniel Apsley, Harold Blum, Mircea Mustaţă, Karen Smith, Kevin Tucker, Joe Waldron for their interest in this work.

\section{Background and conventions}\label{sec:background}

We start by fixing some notation and conventions. All schemes are Noetherian, separated, and excellent. We work over an $F$-finite field of characteristic $p>0$ when discussing test ideals; for multiplier ideals and pseudo-effective thresholds we allow any base field.

Let $X$ be a normal integral scheme.
We fix a canonical divisor $K_X$ once and for all.

A \emph{pair} $(X,\Delta)$ means $X$ normal and $\Delta\ge 0$ a $\Q$-divisor such that $K_X+\Delta$ is $\Q$-Cartier.
We let $\mathrm{ind}(K_X+\Delta)$ denote the Cartier index, i.e.\ the smallest $m>0$ such that $m(K_X+\Delta)$ is Cartier.

\subsection{Frobenius}

Let $X$ be an $F$-finite scheme of characteristic $p>0$.
Write $\Frob^e:X\to X$ for the $e$-fold absolute Frobenius.
We denote by $\Frobpush{e}\mathscr{M}$ the pushforward of an $\cO_X$-module $\mathscr{M}$.

If $X$ is normal and $\omega_X$ is the dualizing sheaf, one has a Frobenius trace map
\[
\Tr^e_X:\Frobpush{e}\omega_X \to \omega_X,
\]
adjoint to the natural inclusion $\cO_X\hookrightarrow \Frobpush{e}\cO_X$ under Grothendieck duality.

We request the reader to consult \cite{FBook} for a lucid guide to the world ruled by $Frobenius$ and \cite{BSTZ10,ST} for more technical terms.

\begin{definition}
     By a triple $(X, \Delta, \fa^t)$, $t \in \R_{\geq 0}$ we mean a normal integral scheme $X$ together with an effective $\Q$-divisor $\Delta$ on $X$ such that $K_X + \Delta$ is $\Q$-Cartier, and an ideal $\fa \subseteq \cO_X$.
\end{definition}

\begin{definition}(\cite{ST,FBook})\label{Def-testideal}Let $(X, \Delta, \fa^t)$ is a triple where $X = \Spec R$. The test ideal $\tau(X, \Delta, \fa^t)$ is the unique smallest non-zero ideal $J \subseteq R$ such that for every $e > 0$ and every section $\phi \in \Hom_{R}(F^e_* R(\lceil (p^e - 1)\Delta \rceil), R) \subseteq \Hom_{R}(F^e_* R, R)$, we have $\phi(F^e_* (\fa^{\lceil t (p^e - 1) \rceil} J)) \subseteq J$.
\end{definition}

\begin{definition}(\cite{ST,FBook})\label{def:param-test} Assume the setup and let $\Gamma \geq 0$ be a $\Q$-divisor. Then the test module, $\tau(\omega_R, \Gamma, \fa^{t})$, is the unique smallest nonzero submodule $J \subseteq \omega_X$ such that
for every $e \geq 0$ and every map $\phi \in \Hom_{R}\big(F^e_* (\omega_R( \lceil (p^e - 1) \Gamma \rceil)), \omega_R\big)$
 we have that $\phi\big(F^e_* ( \fa^{\lceil t(p^e - 1) \rceil}\cdot J)\big) \subseteq J.$
When $\Gamma = 0$ (or $\fa = R$) we write, $\tau(\omega_R,\fa^{t})$ (or $\tau(\omega_R, \Gamma)$, respectively).
\end{definition}

\begin{lemma}({\cite[Lemma~4.2]{ST}})\label{lem:module-to-ideal}
Let $X$ be normal and affine of characteristic $p>0$, and fix a canonical divisor $K_X$ such that $-K_X$ is effective. Let $\Delta\ge 0$ with $K_X+\Delta$ $\Q$-Cartier of index not divisible by $p$.
Then
\[
\tau(\omega_X , K_X+\Delta,\mathfrak{a}^t) \;=\; \tau(X,\Delta;\mathfrak{a}^t).
\]

\end{lemma}

When $R$ is normal but not $\Q$-Gorenstein, one can still define a meaningful test ideal invariant: the \emph{big test ideal} $\taub$.
We will use Schwede's sum formula \cite{SchwedeBig} as the definition for our purposes.

\begin{thm}(\cite[Theorem~6.2]{SchwedeBig})\label{def:big-test}
Let $R$ be an $F$-finite normal domain of characteristic $p>0$, and let $\mathfrak{a}\subseteq R$ be a nonzero ideal.
For $t\in\R_{\ge 0}$, the \emph{big test ideal} $\taub(R;\mathfrak{a}^t)$ (\cite{HH}, \cite[Definition 3.16]{SchwedeBig}) is given by the sum
\[
\taub(R;\mathfrak{a}^t)
\;=\;
\sum_{\Delta} \tau(R,\Delta;\mathfrak{a}^t),
\]
where the sum ranges over all effective $\Q$-divisors $\Delta$ on $X=\Spec R$
such that $K_X+\Delta$ is $\Q$-Cartier with index not divisible by $p$.
\end{thm}

\begin{remark}
Schwede proves that this sum agrees with the intrinsic definition of $\taub$ \cite{SchwedeBig}.
For our applications, the sum formula is all that we need.
\end{remark}

\subsection{Multiplier ideals}\label{def:Jpi}

Let $X$ be a normal affine variety over a field (any characteristic).
Fix an effective $\Q$-divisor $\Delta$ with $K_X+\Delta$ $\Q$-Cartier, and a nonzero ideal $\mathfrak{a}\subseteq \cO_X$.
Choose a log resolution\footnote{whose existence we will assume; however all the examples in this note will have explicit resolution given by blow up.} $\pi:Y\to X$, so $\mathfrak{a}\cO_Y=\cO_Y(-G)$ for an effective Cartier divisor $G$ on $Y$,
and $Y$ is regular with $\Supp(\Exc(\pi)\cup G)$ simple normal crossings.

For $t\in\R_{\ge 0}$ define
\[
\cJ(X,\Delta;\mathfrak{a}^t)
\;:=\;
\pi_*\cO_Y\Bigl(\ceil{K_Y-\pi^*(K_X+\Delta)-tG}\Bigr)\subseteq \cO_X.
\]
 We refer the reader to \cite{Positivity-II} for a comprehensive introduction and \cite{dFH,dFDTT} for a more technical and interesting view.

\section{$\tau$ equals $\cJ$ for large exponents, following \cite{ST}}\label{sec:asymptotic}

We prove a general theorem asserting that for a fixed log resolution $\pi:Y\to X$ obtained by blowing up of $\mathfrak{a}$,
the test ideal $\tau(X,\Delta;\mathfrak{a}^t)$ coincides with the multiplier ideal $\cJ(X,\Delta;\mathfrak{a}^t)$
for all sufficiently large $t$. The main idea and ingredients are all in \cite{ST}, but we write down the exact version we need with streamlined proof for the ease of the reader.

\begin{setup}\label{setup:asymptotic}
Let \(X=\Spec R\) be normal, affine, integral, and
\(F\)-finite of characteristic \(p>0\).  Let \(\Delta\ge 0\) be a
\(\Q\)-divisor such that \(K_X+\Delta\) is \(\Q\)-Cartier, and let
\(\mathfrak a\subset R\) be a nonzero ideal.  Assume that the normalized
blowup \(\pi:Y\to X\) of \(\mathfrak a\) is regular, that
\(\mathfrak a\cO_Y=\cO_Y(-G)\) for an effective Cartier divisor \(G\), that
\(-G\) is \(\pi\)-ample, and that
\(\Supp(G\cup \Exc(\pi)\cup \pi_*^{-1}\Delta)\) is simple normal crossings.

Write the Cartier index of $K_X+\Delta$ as
\[
\mathrm{ind}(K_X+\Delta)=p^b m,\qquad p\nmid m.
\]
Choose an integer $c>0$ such that $m \mid (p^c-1)$.
Equivalently, $(p^c-1)p^b(K_X+\Delta)$ is Cartier.
\end{setup}

\begin{definition}\label{def:allowed-frac}
With $b,c$ as above, define the finite set
\[
A_{b,c} \;:=\; \Bigl\{\alpha\in[0,1)\cap\Q \ \Big|\ p^b(p^c-1)\alpha\in\Z\Bigr\}.
\]
Equivalently, $A_{b,c}=\{0,\frac{1}{p^b(p^c-1)},\dots,\frac{p^b(p^c-1)-1}{p^b(p^c-1)}\}$.
\end{definition}

We first prove the ``asymptotic equality" theorem for exponents $t$ of the form $t=n+\alpha$ with $n\in\Z_{\ge 0}$ large and $\alpha\in A_{b,c}$.

On a regular scheme with simple normal crossings data, parameter test modules have an explicit description.

\begin{lemma}(\cite{Hara-Yoshida,Takagi-multiplier,Blickle-toric})\label{lem:snc}
Let $Y$ be regular, $F$-finite, and of characteristic $p>0$.
Let $\Theta$ be a $\Q$-divisor on $Y$ with simple normal crossings support such that $(p^e-1)\Theta$ is Cartier for some $e>0$.
Then
\[
\tau(\w_Y ,\Theta) = \cO_Y\bigl(\ceil{K_Y-\Theta}\bigr)\subseteq \omega_Y.
\]
More generally, if $\mathfrak{b}\subseteq \cO_Y$ is an ideal such that $\mathfrak{b}\cO_Y=\cO_Y(-H)$ for a Cartier divisor $H$ whose support is SNC with $\Theta$,
then for every real $t \ge 0$
\[
\tau(\w_Y, \Theta,\mathfrak{b}^t) \;=\; \cO_Y\bigl(\ceil{K_Y-\Theta-tH}\bigr).
\]
\end{lemma}

We now specialize \cite[Section 5]{ST} to our fixed resolution $\pi:Y\to X$.
We keep Setup~\ref{setup:asymptotic} throughout.

\begin{definition}\label{def:GammaY}
For $t\in\R_{\ge 0}$ define the $\Q$-divisor on $Y$
\[
\Gamma_Y(t) \;:=\; \pi^*(K_X+\Delta) + tG.
\]
Define also the integral divisor
\[
E_\pi(t) \;:=\; \ceil{K_Y - \Gamma_Y(t)}=\ceil{K_Y-\pi^*(K_X+\Delta)-tG}.
\]
\end{definition}

\begin{remark}
The rounding identity
\[
E_\pi(n+\alpha) = E_\pi(\alpha) - nG
\qquad(n\in\Z,\ 0\le \alpha<1)
\]
holds because $G$ is integral Cartier.
\end{remark}

Let us first treat the case where $K_X+\Delta$ has index not divisible by $p$ (i.e.\ $b=0$).
Then $(p^c-1)(K_X+\Delta)$ is Cartier for some $c>0$.
In that situation, \cite{ST} gives a concrete description of $\tau(\omega_X , K_X+\Delta,\mathfrak{a}^t)$
as the stable image as follows.

\begin{definition}\label{def:PhiY}
Assume for the moment that $b=0$, so $(p^c-1)(K_X+\Delta)$ is Cartier for some fixed $c>0$.
For $t\in \Q_{\ge 0}$ with $(p^c-1)t\in\Z$, set
\[
M(t) \;:=\; \tau(\w_Y,\Gamma_Y(t)) \subseteq \omega_Y.
\]
Define the $\cO_Y$-linear map
\[
\Phi^c_{Y,t}:\Frobpush{c}\Bigl(M(t)\otimes \cO_Y\bigl((1-p^c)\Gamma_Y(t)\bigr)\Bigr)\longrightarrow M(t)
\]
to be the restriction of the trace map $\Tr^c_Y$ (twisted appropriately) as in \cite[\S5]{ST}.
Let
\[
N(t):=\ker(\Phi^c_{Y,t}).
\]
\end{definition}

\begin{lemma}(\cite[Proof of Theorem 5.1]{ST})\label{lem:stabilize}
Assume $b=0$ and $(p^c-1)t\in\Z$.
Then $\tau(\omega_X , K_X+\Delta,\mathfrak{a}^t)$ is the stable image of the sequence
\[
\pi_*M(t)\xleftarrow{\ \pi_*\Phi^c_{Y,t}\ }\Frobpush{c}\bigl(\pi_*M(t)\bigr)\xleftarrow{\ \Frobpush{c}(\pi_*\Phi^c_{Y,t})\ }\Frobpush{2c}\bigl(\pi_*M(t)\bigr)\xleftarrow{\ \cdots\ } \cdots
\]

In particular, if $\pi_*\Phi^c_{Y,t}$ is surjective, then
\[
\tau(\omega_X , K_X+\Delta,\mathfrak{a}^t) = \pi_*M(t) \subseteq \omega_X.
\]
\end{lemma}

\begin{lemma}\label{lem:twist-kernel}
Assume $b=0$ and $t=n+\alpha$ with $n\in\Z_{\ge 0}$ and $\alpha\in A_{0,c}$.
Then there are natural isomorphisms
\[
M(t) \;\cong\; M(\alpha)\otimes \cO_Y(-nG),
\qquad
N(t) \;\cong\; N(\alpha)\otimes \cO_Y(-nG).
\]
\end{lemma}

\begin{proof}
Because $G$ is Cartier and integral, adding $nG$ to $\Gamma_Y(\alpha)$ simply twists:
\[
\Gamma_Y(n+\alpha)=\Gamma_Y(\alpha)+nG.
\]
Then, by \cite[Definition 4.3]{ST},
\[
\tau(\w_Y,\Theta + nG)=\tau(\w_Y,\Theta)\otimes \cO_Y(-nG)
\]
for any $\Q$-divisor $\Theta$ and any Cartier divisor $nG$.
The rest follows trivially, using projection formula.
\end{proof}

\begin{lemma}(\cite[\S5]{ST})\label{lem:serre}
Fix $b=0$ and $c>0$.
For each $\alpha\in A_{0,c}$ there exists an integer $n_\alpha\ge 0$ such that
\[
R^1\pi_*\bigl(N(\alpha)\otimes \cO_Y(-nG)\bigr)=0
\qquad\text{for all }n\ge n_\alpha.
\]
Consequently, if $t=n+\alpha$ with $n\ge n_\alpha$, then $\pi_*\Phi^c_{Y,t}$ is surjective.
\end{lemma}

\begin{proof}
Since $N(\alpha)$ is coherent, and since $-G$ is $\pi$-ample, by relative Serre vanishing,
\[
R^i\pi_*\bigl(N(\alpha)\otimes \cO_Y(-nG)\bigr)=0
\qquad\text{for all }i>0\text{ and }n\gg 0.
\]

This gives the existence of $n_\alpha$ such that $R^1$ vanishes for $n\ge n_\alpha$.

Now take the exact sequence
\[
0\to N(t)\to \Frobpush{c}\bigl(M(t)\otimes\cO_Y((1-p^c)\Gamma_Y(t))\bigr)\xrightarrow{\Phi^c_{Y,t}} M(t)\to 0
\]
and apply $\pi_*$.
The vanishing of $R^1\pi_*N(t)$ implies exactness on global pushforwards, hence $\pi_*\Phi^c_{Y,t}$ is surjective.
Finally use Lemma~\ref{lem:twist-kernel} to identify $N(t)\cong N(\alpha)\otimes\cO_Y(-nG)$.
\end{proof}

We are ready for the prime-to-$p$ index asymptotic comparison. One should view it as ``\cite[Theorem 5.1]{ST} with $\alpha\in A_{0,c}$ varying".

\begin{theorem}(cf. \cite[Theorem 5.1]{ST})\label{thm:asymptotic-b0}
In Setup~\ref{setup:asymptotic}, assume $b=0$.
Fix $c>0$ such that $(p^c-1)(K_X+\Delta)$ is Cartier.
Then there exists an integer $T\ge 0$ such that for every $t=n+\alpha$ with $\alpha\in A_{0,c}$ and $n\ge T$ one has
\[
\tau(X,\Delta;\mathfrak{a}^t) \;=\; \cJ(X,\Delta;\mathfrak{a}^t)
\;=\;
\pi_*\cO_Y\bigl(E_\pi(t)\bigr).
\]
\end{theorem}

\begin{proof}
Fix $\alpha\in A_{0,c}$.
Choose $n_\alpha$ as in Lemma~\ref{lem:serre}, and set $T:=\max_{\alpha\in A_{0,c}} n_\alpha$.
Then for $t=n+\alpha$ with $n\ge T$, $\pi_*\Phi^c_{Y,t}$ is surjective.
Lemma~\ref{lem:stabilize} therefore gives
\[
\tau(\omega_X , K_X+\Delta,\mathfrak{a}^t)=\pi_*M(t)\subseteq \omega_X.
\]

Next, since $Y$ is regular and \(\Supp(G\cup \Exc(\pi)\cup \pi_*^{-1}\Delta)\) is SNC, Lemma~\ref{lem:snc} applies to $\Gamma_Y(t)$:
\[
M(t)=\tau(\w_Y,\Gamma_Y(t))=\cO_Y\bigl(\ceil{K_Y-\Gamma_Y(t)}\bigr)=\cO_Y(E_\pi(t)).
\]
Thus $\pi_*M(t)=\pi_*\cO_Y(E_\pi(t)) \subseteq \cO_X$.

Finally, apply Lemma~\ref{lem:module-to-ideal}:
\[
\tau(X,\Delta;\mathfrak{a}^t)
= \pi_*\cO_Y(E_\pi(t)) = \cJ(X,\Delta;\mathfrak{a}^t).
\] by Definition~\ref{def:Jpi}.
\end{proof}

We now handle the general case $b\ge 0$ in Setup~\ref{setup:asymptotic}.

\begin{lemma}(\cite[Lemma~4.4(b)]{ST})\label{lem:frob-scaling}
Let $X$ be normal, affine, and $F$-finite of characteristic $p>0$.
Let $\Gamma$ be a $\Q$-divisor and $\mathfrak{a}\subseteq \cO_X$ a nonzero ideal.
For any integer $b\ge 1$ and any $t\ge 0$ one has
\[
\tau(\w_X,\Gamma,\mathfrak{a}^t)
=
\Tr^b_X\Bigl(\Frobpush{b}\tau(\w_X,p^b\Gamma,\mathfrak{a}^{p^b t})\Bigr)\subseteq \omega_X,
\]
\end{lemma}

\begin{lemma}(\cite[\S 2.3]{BST-alt})\label{lem:trace-compat}
Let $\pi:Y\to X$ be a proper morphism of normal $F$-finite schemes of characteristic $p>0$, and let $e\ge 1$.
Write $F_X^e$ and $F_Y^e$ for the $e$-fold Frobenius morphisms.
Then $F_X^e\circ \pi=\pi\circ F_Y^e$, hence there is a natural identification
$F_{X*}^e\pi_* \omega_Y \cong \pi_*F_{Y*}^e\omega_Y$.
With respect to this identification, the trace maps satisfy
\[
\Tr_X^e\circ F_{X*}^e(\Tr_\pi)
\;=\;
\Tr_\pi\circ \pi_*(\Tr_Y^e),
\]
as maps $F_{X*}^e\pi_*\omega_Y\to \omega_X$.

\end{lemma}

I thank my advisor Karl Schwede for suggesting the following Theorem \ref{thm:asymptotic-general}.

\begin{theorem}(cf. \cite[Theorem 5.1]{ST})\label{thm:asymptotic-general}
In Setup~\ref{setup:asymptotic}, let $b\ge 0$ be such that $p^b(K_X+\Delta)$ has Cartier index not divisible by $p$,
and choose $c>0$ such that $(p^c-1)p^b(K_X+\Delta)$ is Cartier.
Then there exists an integer $T\ge 0$ such that for every $t=n+\alpha$ with $\alpha\in A_{b,c}$ and $n\ge T$ one has
\[
\tau(X,\Delta;\mathfrak{a}^t) \;=\; \cJ(X,\Delta;\mathfrak{a}^t).
\]
\end{theorem}

\begin{proof}
Set $\Gamma:=K_X+\Delta$.
By assumption $\mathrm{ind}(\Gamma)=p^b m$ with $p\nmid m$, and we chose $c>0$ so that $(p^c-1)p^b\Gamma$ is Cartier.

Fix $t=n+\alpha$ with $\alpha\in A_{b,c}$ and $n\gg 0$ (to be made precise below). Theorem~\ref{thm:asymptotic-b0} applied to $p^b\Gamma$ gives existence of $T$ so that for all such $t$ with $n\ge T$,
\[
\tau\bigl(\w_X, p^b\Gamma,\mathfrak{a}^{p^b t}\bigr)
=
\pi_*\omega_Y\Bigl(\ceil{K_Y-\pi^*(p^b\Gamma)-p^b t\,G}\Bigr)
\subseteq \omega_X.
\]
On $Y$ we have $\pi^*(p^b\Gamma)=p^b\pi^*\Gamma$, and by Definition~\ref{def:GammaY},
$p^b\Gamma_Y(t)=p^b\pi^*\Gamma + p^b t\,G$.
Since $Y$ is regular and $\Supp(\Exc(\pi)\cup G)$ is SNC, Lemma~\ref{lem:snc} gives
\[
\omega_Y\Bigl(\ceil{K_Y-\pi^*(p^b\Gamma)-p^b t\,G}\Bigr)
=
\tau\bigl(\w_Y, p^b\Gamma_Y(t)\bigr).
\]
Hence 
\[
\tau\bigl(\w_X, p^b\Gamma,\mathfrak{a}^{p^b t}\bigr)
=
\pi_*\tau\bigl(\w_Y, p^b\Gamma_Y(t)\bigr)
\subseteq \omega_X.
\]

Now apply Lemma~\ref{lem:frob-scaling} and Lemma~\ref{lem:trace-compat} to get:
\[
\tau(\w_X,\Gamma,\mathfrak{a}^t)
=
\Tr_X^b\Bigl(\Frobpush{b}\pi_*\tau(\w_Y,p^b\Gamma_Y(t))\Bigr)
= \pi_*(\tau(\w_Y,\Gamma_Y(t))).
\]

Applying Lemma~\ref{lem:snc} and Lemma~\ref{lem:module-to-ideal} gives
\[
\tau(X,\Delta;\mathfrak{a}^t)
=
\pi_*\cO_Y(E_\pi(t))
=
\cJ(X,\Delta;\mathfrak{a}^t),
\]
as desired.
\end{proof}

\subsection{Extending Theorem~\ref{thm:asymptotic-general} to all real $t\gg 0$}
In many applications (including ours), one wants the equality for \emph{all} real $t$ sufficiently large.
The idea is simple: both $t\mapsto \tau(X,\Delta;\mathfrak{a}^t)$ and $t\mapsto \cJ(X,\Delta;\mathfrak{a}^t)$ are right-continuous and have discrete jumping sets,
so if two such step functions agree on a sufficiently fine lattice in each unit interval beyond some point, they agree everywhere beyond that point.

\begin{theorem}
    
\label{thm:asymptotic-real}
In Setup~\ref{setup:asymptotic}, there exist integer $T_0\ge 0$ such that
\[
\tau(X,\Delta;\mathfrak{a}^t)=\cJ(X,\Delta;\mathfrak{a}^t)
\qquad\text{for all real }t\ge T_0.
\]
\end{theorem}

\begin{proof}
Suppose $\fa$ is generated by $r$ elements.  Set $T_0:=\max\{T,r-1\}$. By the Skoda's theorem for test ideals \cite{ST} and multiplier ideals \cite{LazPosI}, for all real $t\ge r$ one has
\[
\tau(X,\Delta;\fa^{t+1})=\fa\cdot \tau(X,\Delta;\fa^t)
\qquad\text{and}\qquad
\cJ(X,\Delta;\fa^{t+1})=\fa\cdot \cJ(X,\Delta;\fa^t).
\]
Hence every jumping number of either family in $(r-1,\infty)$ is an integer translate of a jumping number in $(r-1,r]$.

\medskip

By discreteness and rationality of jumping numbers for test ideals and multiplier ideals \cite{ST,LazPosI},
there are finitely many jumping numbers of $\tau(X,\Delta;\fa^t)$ and $\cJ(X,\Delta;\fa^t)$ in the interval $(r-1,r]$,
and all of them are rational . Let $d$ be the least common multiple of the denominators of all these finitely many jumping numbers.
Then every jumping number of either family in $(r-1,\infty)$ lies in $\frac1d\mathbb Z$.

Choose $b,c$ so that $d\mid p^b(p^c-1)$ and $(p^c-1)p^b(K_X+\Delta)$ is Cartier.
Then
\[
A_{b,c}+\mathbb Z \;=\;\frac{1}{p^b(p^c-1)}\mathbb Z
\supseteq \frac1d\mathbb Z.
\]
By Theorem~\ref{thm:asymptotic-general} , for all $t\in A_{b,c}+\mathbb Z$ with $t\ge T_0$ we have
$\tau(X,\Delta;\fa^t)=\cJ(X,\Delta;\fa^t)$.

\medskip

Fix any real $t\ge T_0$ and set $t^-:=\lfloor dt\rfloor/d\in \frac1d\mathbb Z$.
Then $t^-\le t<t^-+1/d$, and since all jumping numbers of both families in $[T_0,\infty)$ lie in $\frac1d\mathbb Z$,
there is \emph{no} jump of either family in the interval $(t^-,t]$.
Therefore both families are constant on $(t^-,t]$ (by right-continuity, \cite[Lemma 6.1]{ST}), and hence
\[
\tau(X,\Delta;\fa^t)=\tau(X,\Delta;\fa^{t^-})
\quad\text{and}\quad
\cJ(X,\Delta;\fa^t)=\cJ(X,\Delta;\fa^{t^-}).
\]
But $t^-\in \frac1d\mathbb Z\subseteq A_{b,c}+\mathbb Z$ and $t^-\ge T_0$, so Theorem~\ref{thm:asymptotic-general} gives
$\tau(X,\Delta;\fa^{t^-})=\cJ(X,\Delta;\fa^{t^-})$, and thus $\tau(X,\Delta;\fa^t)=\cJ(X,\Delta;\fa^t)$.
\end{proof}

\begin{remark}
Theorem \ref{thm:asymptotic-real}
fails badly beyond Setup~\ref{setup:asymptotic} in general. We note from literature a few example of such failure in  \cite[Theorem 3.1, Example 3.4]{MY-vs} and \cite[\S 4]{Mus-TW} and \cite[Chapter 4, \S3]{FBook}, punched with Skoda's Theorem. 
\end{remark}

\begin{corollary}\label{cor:pair-asymptotic-ray}
In Setup~\ref{setup:asymptotic}, and let $\mathfrak a\subseteq R$ be an ideal generated by $r$ elements.
Fix a real number $t\ge r$.
Then there exists an integer $N_0\ge 0$ such that for all integers $N\ge N_0$ one has
\[
\tau_b(X;\mathfrak a^{t+N}) \;=\; \cJ_{\mathrm{dFH}}(X;\mathfrak a^{t+N}).
\]
\end{corollary}

\begin{proof}
By \cite[Cor.~5.4]{SchwedeBig} there exists a finite set of effective boundaries $J_\tau(t)$ with $K_X+\Delta$ $\Q$-Cartier such that
\[
\tau_b(X;\mathfrak a^t)=\sum_{\Delta\in J_\tau(t)} \tau(X,\Delta;\mathfrak a^t).
\]
Likewise, by Noetherianity there exists a finite set $J_{\cJ}(t)$ such that
\[
\cJ_{\mathrm{dFH}}(X;\mathfrak a^t)=\sum_{\Delta\in J_{\cJ}(t)} \cJ(X,\Delta;\mathfrak a^t).
\]
Set $J:=J_\tau(t)\cup J_{\cJ}(t)$; enlarging the index set does not change either sum.

For each $\Delta\in J$, apply Theorem~\ref{thm:asymptotic-real} to obtain $T_\Delta$ such that
$\tau(X,\Delta;\mathfrak a^u)=\cJ(X,\Delta;\mathfrak a^u)$ for all $u\ge T_\Delta$.
Let $T=\max_{\Delta\in J}T_\Delta$.

Choose $N_0$ so that $t+N_0\ge T$.
For $N\ge N_0$, Skoda's theorem for test ideals and multiplier ideals (\cite{LazPosI,BSTZ10}) yields, for every $\Delta\in J$,
\[
\tau(X,\Delta;\mathfrak a^{t+N})=\mathfrak a^N\tau(X,\Delta;\mathfrak a^t),
\qquad
\cJ(X,\Delta;\mathfrak a^{t+N})=\mathfrak a^N\cJ(X,\Delta;\mathfrak a^t),
\]
since $t\ge r$.
Therefore the same finite set $J$ computes both pair ideals at $t+N$:
\[
\tau_b(X;\mathfrak a^{t+N})=\sum_{\Delta\in J}\tau(X,\Delta;\mathfrak a^{t+N}),
\qquad
\cJ_{\mathrm{dFH}}(X;\mathfrak a^{t+N})=\sum_{\Delta\in J}\cJ(X,\Delta;\mathfrak a^{t+N}).
\]
As $t+N\ge T$, the summands coincide termwise, hence the sums are equal. 
\end{proof}

For the cone argument in Section~\ref{sec:irrational-f-jumps} we will need the following well-known \emph{upper} bound on test ideals in terms of multiplier ideals $\cJ$.

\begin{lemma}(\cite[Theorem 2.13]{Takagi-multiplier}, \cite{BST-alt})\label{lem:tau-in-J}
Let $X=\Spec R$ be normal, and $F$-finite of characteristic $p>0$.
Let $\Delta\ge 0$ be such that $K_X+\Delta$ is $\Q$-Cartier.
Let $\mathfrak{a}\subseteq R$ be an ideal.
Then for every real $t\ge 0$,
\[
\tau(X,\Delta;\mathfrak{a}^t)\subseteq \cJ(X,\Delta;\mathfrak{a}^t).
\]
\end{lemma}

\section{Computation on cones}\label{sec:cone}

In this section we explain why an irrational pseudo-effective threshold on a smooth variety produces irrational jumping behavior
for multiplier/test ideals on the cone over that variety. This section is influenced by \cite{Urbinati} which is my main reference for this section as well. We follow \cite{EGAIV,Kollar-Sing} as general reference on cones.

Let $W$ be a smooth projective variety over an infinite field $k$.
Let $A$ be an ample Cartier divisor on $W$.
Define the section ring
\[
R(W,A) := \bigoplus_{m\ge 0} H^0\bigl(W,\cO_W(mA)\bigr),
\]
and the affine cone
\[
C := \Spec R(W,A).
\]
Let $\m\subset R(W,A)$ be the homogeneous maximal ideal, corresponding to the vertex of the cone.

\begin{lemma}(\cite{EGAIV})\label{prop:blowup-vertex}
Assume $R(W,A)$ is generated in degree $1$.
Let $\pi:Y\to C$ be the blowup of $C$ at the vertex $\m$, and \(W_0\subset Y\) be the
exceptional divisor with $\m\cO_Y=\cO_Y(-W_0)$.
Then:
\begin{enumerate}[label=(\alph*),leftmargin=2.4em]
  \item There is an isomorphism
\[
Y\cong \Tot\bigl(\cO_W(-A)\bigr)
\]
over \(W\).  Let \(p:Y\to W\) denote this bundle projection.

\item The exceptional divisor \(W_0\) is the zero section.  In particular,
\[
\cO_Y(-W_0)\cong p^*\cO_W(A).
\]

\item For every integer \(n\ge 0\),
\[
\pi_*\cO_Y(-nW_0)=\m^n.
\]
Moreover, for every integer \(n\ge 0\),
\[
\pi_*\cO_Y(nW_0)=\cO_C.
\]

\item The canonical bundle of \(Y\) is
\[
\omega_Y\cong p^*(\omega_W\otimes \cO_W(A)),
\]
so after choosing canonical divisors compatibly we may write
\[
K_Y=p^*K_W-W_0.
\]
\end{enumerate}
\end{lemma}

\subsection{Pseudo-effective thresholds}\label{def:t0}\label{def:theta}

(cf.\cite[Proof of Theorem 3.3]{Urbinati}) Let $W$ be a smooth projective surface and $A$ ample. Define \emph{pseudo-effective threshold} as
\[
\vartheta(W,A):=\inf\{s\in \R\mid sA-K_W\in \overline{\Eff}(W)\}.
\]
Since \(A\) is ample, \(\vartheta(W,A)\) is always finite.
We also define the truncated version,
\[
t_0(W,A):=\inf\{t\ge 0\mid tA-K_W\in \overline{\Eff}(W)\}.
\]
Note \(t_0(W,A)=\max\{\vartheta(W,A),0\}\).

\begin{remark}\label{lem:theta=t0}
If \(K_W\) is pseudo-effective, then \(\vartheta(W,A)=t_0(W,A)\). In all the examples discussed, \(K_W\) is indeed pseudo-effective, giving
\(\vartheta=t_0\).  Nevertheless, the distinction matters as shown in the following Example 4.4.
\end{remark}
\begin{remark}\label{rem:veronese}
Replacing $A$ by a positive multiple $dA$ does not change the cone up to a Veronese regrading.
The irrationality of pseudo-effective threshold $\vartheta(W,A)$ is unchanged because:
\[
\vartheta(W,dA) = \inf\{t\mid t(dA)-K_W\in\psef(W)\} = \frac{1}{d}\vartheta(W,A).
\]
In practice we choose $d$ large enough so that $R(W,dA)$ is generated in degree $1$
(and even projectively normal); see \cite[\S1.8]{LazPosI}. Henceforth, we assume $R(W,A)$ is \emph{generated in degree $1$}.
\end{remark}

\begin{exam}
    Take \(W=\mathbf P^2\) and \(A=H\), the hyperplane class.  Then
\(K_W=-3H\), so
\[
\vartheta(W,A)=\inf\{s\in\R\mid (s+3)H\in\overline{\Eff}(W)\}=-3,
\qquad
 t_0(W,A)=0.
\]
The cone over \((\mathbf P^2,H)\) is the smooth affine threefold \(\A^3\), and
its multiplier ideals 
\[
\cJ_{\mathrm{dFH}}(C;\m^t)=\m^{\max\{0,\lfloor t+1-3\rfloor\}}
=\m^{\max\{0,\lfloor t-2\rfloor\}},
\]
\end{exam}

We record the following very well-known lemma. 
\begin{lemma}(cf. \cite{Kollar-Sing,Goto-Watanabe-1}, \cite[Lemma 2.32]{BdFF})\label{prop:cone-QG}
The cone $C$ is $\Q$-Gorenstein if and only if $K_W$ is $\Q$-linearly equivalent to a rational multiple of $A$.
In that case, $\vartheta(W,A)\in\Q$.
\end{lemma}

\begin{corollary}(\cite[Lemma 2.32]{BdFF})\label{cor:cone-not-QG}
If $\vartheta(W,A)\notin\Q$, then the cone $C$ is not $\Q$-Gorenstein. The cone \(C\) is not numerically \(\Q\)-Gorenstein as well.
\end{corollary}

\begin{remark}\label{rem:num-q-factorial}
In \cite[Example~5.6]{BdFFU}, the authors show that
\(\mathrm{Cl}_{\mathrm{num}}(C)\cong \NS(W)/\Z[A]\).  In particular, the cone is numerically
\(\Q\)-factorial if and only if \(\rho(W)=1\).  Since all of our examples come from varieties with Picard number at least 2, they are also not numerically \(\Q\)-factorial.
\end{remark}

\begin{definition}\label{def:qm}
For each integer \(m\ge 1\), define
\[
q_m:=\min\{q\in\Z\mid H^0\bigl(W,\cO_W(qA-mK_W)\bigr)\ne 0\}.
\]
\end{definition}

\begin{lemma}\label{lem:qm-limit}
The sequence \((q_m)_{m\ge 1}\) is subadditive and satisfies
\[
\lim_{m\to\infty}\frac{q_m}{m}=\vartheta(W,A).
\]
\end{lemma}

\begin{proof}
If \(s_m\in H^0(W,\cO_W(q_mA-mK_W))\setminus\{0\}\) and
\(s_n\in H^0(W,\cO_W(q_nA-nK_W))\setminus\{0\}\), then the product
\(s_ms_n\) is a nonzero section of
\(\cO_W((q_m+q_n)A-(m+n)K_W)\).  Hence
\(q_{m+n}\le q_m+q_n\).  By Fekete's lemma\footnote{Thanks to Daniel Apsley for pointing this out!},
\(\lim_{m\to\infty} q_m/m = \inf_m q_m/m\) exists in \(\R\cup\{-\infty\}\).
Since \(A\) is ample, the values are bounded below linearly, so the limit is
finite.

If \(H^0(W,\cO_W(q_mA-mK_W))\ne 0\), then \(q_mA-mK_W\) is linearly equivalent
to an effective divisor, hence pseudo-effective.  Therefore
\(q_m/m\ge \vartheta(W,A)\) for every \(m\), and so
\[
\lim_{m\to\infty}\frac{q_m}{m}\ge \vartheta(W,A).
\]

Conversely fix any rational number \(s>\vartheta(W,A)\).  Choose a rational
number \(r\) with \(\vartheta(W,A)<r<s\).  Then \(rA-K_W\) is pseudo-effective,
while \((s-r)A\) is ample; hence
\[
sA-K_W=(s-r)A+(rA-K_W)
\]
is big.  For all sufficiently divisible integers \(m\), the divisor
\(m(sA-K_W)\) is Cartier and has a nonzero section.  Consequently
\(q_m\le ms\) for all such \(m\).  Dividing by \(m\) and letting \(m\to\infty\)
shows
\[
\limsup_{m\to\infty}\frac{q_m}{m}\le s.
\]
Since this holds for every rational \(s>\vartheta(W,A)\), we conclude that
\(\limsup q_m/m\le \vartheta(W,A)\).  Together with the first inequality this
proves the claim.
\end{proof}

\subsection{Boundaries}

We now explain why \emph{any} boundary $\Gamma$ on the cone with $K_C+\Gamma$ $\Q$-Cartier must have a coefficient constraint controlled by $\vartheta(W,A)$.

\begin{lemma}(\cite{Urbinati})\label{lem:k-ineq}
Let $C$ be the cone over $(W,A)$ as above, and let $\pi:Y\to C$ be the blowup of the vertex with exceptional divisor $W_0\cong W$.
Let $\Gamma\ge 0$ be a $\Q$-divisor on $C$ such that $K_C+\Gamma$ is $\Q$-Cartier.
Write
\[
\pi^*(K_C+\Gamma) = K_Y + \Gamma_Y + kW_0,
\]
where $\Gamma_Y$ is an effective $\Q$-divisor on $Y$ not containing $W_0$ and $k\in\Q$.
Then
\[
k-1 \ \ge\ \vartheta(W,A).
\]
\end{lemma}

\begin{proof}
Restrict the equality to $W_0$.
By adjunction on the smooth divisor $W_0\subset Y$,
\[
(K_Y+W_0)|_{W_0} = K_{W_0} = K_W.
\]
Also $\Gamma_Y|_{W_0}$ is effective.
Since $W_0|_{W_0}$ is the normal bundle of $W_0$ in $Y$, Proposition~\ref{prop:blowup-vertex} gives
\[
W_0|_{W_0} \;\sim\; -A.
\]
Restricting $\pi^*(K_C+\Gamma)=K_Y+\Gamma_Y+kW_0$ to $W_0$ and using $\pi(W_0)=\{\m\}$, so $(\pi^*(K_C+\Gamma))|_{W_0}\sim 0$, we obtain
\[
0 \;\sim_\Q\; (K_Y+\Gamma_Y+kW_0)|_{W_0}
\;=\; (K_Y+W_0)|_{W_0} + \Gamma_Y|_{W_0} + (k-1)W_0|_{W_0}
\;=\; K_W + \Gamma_Y|_{W_0} - (k-1)A.
\]
Thus $(k-1)A - K_W \sim_\Q \Gamma_Y|_{W_0}$ is pseudo-effective.
By definition of $\vartheta(W,A)$, this implies $k-1\ge \vartheta(W,A)$.
\end{proof}

Lemma~\ref{lem:k-ineq} shows that the coefficient $k-1$ along $W_0$ is constrained from below by $t_0(W,A)$.
To get jumping behavior, we need boundaries where $k-1$ is \emph{just above} $t_0(W,A)$.

\begin{lemma}\label{lem:boundary-near}
Let $W$ be a smooth projective variety over an infinite field and $A$ ample.
Fix a rational number $s>t_0(W,A)$ such that the divisor class $sA-K_W$ is ample.
Then there exists an effective $\Q$-divisor $\Delta_s$ on $W$ such that:
\begin{enumerate}[label=(\alph*),leftmargin=2.4em]
  \item $K_W+\Delta_s \sim_\Q sA$;
  \item $(W,\Delta_s)$ is klt (in fact, $\Delta_s$ has coefficients $<1$ and SNC support).
\end{enumerate}
\end{lemma}

\begin{proof}
Since $sA-K_W$ is ample, choose an integer $m\gg 0$ such that $m(sA-K_W)$ is very ample Cartier (when $\operatorname{char}k=p>0$, we may additionally assume $p\nmid m$.)
Over an infinite field, Bertini implies a general member $B\in |m(sA-K_W)|$ is smooth (\cite[\href{https://stacks.math.columbia.edu/tag/0FD4}{Tag 0FD4}]{stacks-project}).
Set $\Delta_s:=\frac{1}{m}B$.
Then $K_W+\Delta_s\sim_\Q sA$ and $\Delta_s$ has a single smooth component with coefficient $1/m<1$,
so the pair $(W,\Delta_s)$ is klt.
\end{proof}

From $\Delta_s$ we build a boundary on the cone with coefficient $k=s+1$.

\begin{proposition}(cf. \cite{Urbinati})\label{prop:Gamma-s}
Let $C$ be the cone over $(W,A)$ with blowup $\pi:Y\to C$ and exceptional divisor $W_0$.
Fix $s>t_0(W,A)$ rational and choose $\Delta_s$ as in Lemma~\ref{lem:boundary-near}.
Then there exists an effective $\Q$-divisor $\Gamma_s$ on $C$ such that:
\begin{enumerate}[label=(\alph*),leftmargin=2.4em]
  \item $K_C+\Gamma_s$ is $\Q$-Cartier;
  \item writing $\pi^*(K_C+\Gamma_s)=K_Y+\Gamma_{s,Y}+(s+1)W_0$ with $\Gamma_{s,Y}$ not containing $W_0$,
        we have $\Gamma_{s,Y}|_{W_0}=\Delta_s$;
  \item the coefficients of $\Gamma_{s,Y}$ are $<1$.
\end{enumerate}
\end{proposition}

\begin{proof}
Choose $m\gg 0$ divisible such that both $mA$ and $m(sA-K_W)$ are very ample Cartier divisors.
Pick a smooth member $D\in |m(sA-K_W)|$ and set $\Delta_s:=\frac{1}{m}D$.
Then $0\le \Delta_s<1$ and $K_W+\Delta_s\sim_\Q sA$.

Let $\Gamma_s$ be the divisor on $C$ defined as the projective cone over $\Delta_s$ away from the vertex, see \cite[Claim~3.3.3]{Urbinati}
and the discussion around \cite[Remark~3.3.1]{Urbinati}.
In our notation, this produces an effective $\Q$-divisor $\Gamma_s$ on $C$ whose support avoids the vertex,
such that $K_C+\Gamma_s$ is $\Q$-Cartier and
\[
\pi^*(K_C+\Gamma_s)=K_Y+\Gamma_{s,Y}+(s+1)W_0,
\]
with $\Gamma_{s,Y}|_{W_0}=\Delta_s$.
Since the coefficients of $\Delta_s$ are $<1$, the construction ensures that the coefficients of $\Gamma_{s,Y}$ are also $<1$.
\end{proof}

We can now compute the multiplier ideals $\cJ(C,\Gamma_s;\m^t)$ explicitly.

\begin{proposition}\label{prop:J-Gamma-s}
Let $C$ be the cone over $(W,A)$, and let $\pi:Y\to C$ be the blowup of the vertex with exceptional divisor $W_0$.
Fix a rational $s>t_0(W,A)$ and construct $\Gamma_s$ as in Proposition~\ref{prop:Gamma-s}.
Then for every real $t\ge 0$,
\[
\cJ(C,\Gamma_s;\m^t)=\m^{\floor{s+1+t}}.
\]
\end{proposition}

\begin{proof}
On $Y$, $\m\cO_Y=\cO_Y(-W_0)$ by Proposition~\ref{prop:blowup-vertex}.
By construction,
\[
\pi^*(K_C+\Gamma_s)=K_Y+\Gamma_{s,Y}+(s+1)W_0.
\]
Hence
\[
E_\pi(t)
=
\ceil{K_Y-\pi^*(K_C+\Gamma_s)-tW_0}
=
\ceil{-\Gamma_{s,Y}-(s+1+t)W_0}.
\]
Since the coefficients of $\Gamma_{s,Y}$ are $<1$, we have $\ceil{-\Gamma_{s,Y}}=0$.
Therefore
\[
E_\pi(t)=\ceil{-(s+1+t)W_0}= -\floor{s+1+t}\,W_0.
\]
Finally,
\[
\cJ(C,\Gamma_s;\m^t)
=
\pi_*\cO_Y(E_\pi(t))
=
\pi_*\cO_Y\bigl(-\floor{s+1+t}\,W_0\bigr)
=
\m^{\floor{s+1+t}}
\]
by Proposition~\ref{prop:blowup-vertex}(d).
\end{proof}

In characteristic $p>0$, Proposition~\ref{prop:J-Gamma-s} combines beautifully with the asymptotic comparison theorem from Section~\ref{sec:asymptotic}.

\begin{proposition}\label{prop:tau-Gamma-s}
Assume $\mathrm{char}(k)=p>0$ and $C$ is $F$-finite.
Fix $s>t_0(W,A)$ rational and construct $\Gamma_s$ as in Proposition~\ref{prop:Gamma-s}.
Then there exists an integer $T_s\ge 0$ such that for all real $t\ge 0$ with $\floor{t}\ge T_s$,
\[
\tau(C,\Gamma_s;\m^t)=\m^{\floor{s+1+t}}.
\]
\end{proposition}

\begin{proof}
As $K_C+\Gamma_s$ is $\Q$-Cartier and we are in Setup \ref{setup:asymptotic}, apply Theorem~\ref{thm:asymptotic-real} to get, for $t \ge T_s$,
\[
\tau(C,\Gamma_s;\m^t)=\cJ(C,\Gamma_s;\m^t).
\]
Now apply Proposition~\ref{prop:J-Gamma-s}.
\end{proof}

\subsection{The de Fernex--Hacon multiplier ideal on the cone}

Recall that on a normal variety \(X\), for a Weil divisor \(D\),
de Fernex--Hacon defined the natural pullback by
\[
\cO_Y(-f^{\natural}D):=(\cO_X(-D)\cdot \cO_Y)^{\vee\vee},
\]
whenever \(f:Y\to X\) is a birational morphism with \(Y\) normal
\cite[Definition~2.6]{dFH}.  In our situation \(X=C\), \(f=\pi\), and
\(Y\) is smooth.

Let \(U:=C\setminus \{v\}\), where \(v\) is the vertex.  Since
\(\codim_C\{v\}=d+1\ge 2\), every rank-one reflexive sheaf on \(C\) is
determined by its restriction to \(U\).  On the other hand,
\(U\cong Y\setminus W_0\), and on \(U\) we have
\[
\omega_U\cong p^*\omega_W|_U
\]
because \(\omega_Y\cong p^*(\omega_W\otimes\cO_W(A))\) and
\(\cO_Y(-W_0)|_U\cong \cO_U\).

\begin{proposition}\label{prop:natural-pullback}
For every integer \(m\ge 1\), one has
\[
\cO_Y\bigl(-\pi^{\natural}(mK_C)\bigr)
\cong
p^*\cO_W(-mK_W)(-q_mW_0).
\]
Equivalently,
\[
\pi^{\natural}(mK_C)=p^*(mK_W)+q_mW_0.
\]
\end{proposition}

\begin{proof}
We first compute the graded module of global sections of
\(\cO_C(-mK_C)\).  Since \(\cO_C(-mK_C)\) is a rank-one reflexive sheaf on the
normal affine variety \(C\), and the complement of \(U\) has codimension at
least two, restriction gives
\[
H^0\bigl(C,\cO_C(-mK_C)\bigr)
\cong
H^0(U,\omega_U^{-m}).
\]
Now \(U\to W\) is the complement of the zero section in the total space of the
line bundle \(\cO_W(-A)\), hence it is a principal \(\mathbf G_m\)-bundle.
Consequently
\[
(p|_U)_*\cO_U\cong \bigoplus_{r\in\Z}\cO_W(rA),
\]
and since \(\omega_U^{-m}\cong p^*\cO_W(-mK_W)|_U\), we get
\[
(p|_U)_*\omega_U^{-m}
\cong
\bigoplus_{r\in\Z}\cO_W(rA-mK_W).
\]
Taking global sections yields a canonical graded decomposition
\begin{equation}\label{eq:canonical-module-cone}
H^0\bigl(C,\cO_C(-mK_C)\bigr)
\cong
\bigoplus_{r\in\Z}H^0\bigl(W,\cO_W(rA-mK_W)\bigr).
\end{equation}
By definition of \(q_m\), the right-hand side has no graded pieces below
\(q_m\), and the \(q_m\)-th graded piece is nonzero.

Now consider, for an arbitrary integer \(q\), the invertible sheaf
\[
\mathscr L_q:=p^*\cO_W(-mK_W)(-qW_0)
\]
on \(Y\).  Since \(Y=\Tot(\cO_W(-A))\) and
\(p_*\cO_Y\cong \bigoplus_{n\ge 0}\cO_W(nA)\), the projection formula gives
\[
p_*\mathscr L_q
\cong
\left(\bigoplus_{n\ge 0}\cO_W(nA)\right)\otimes \cO_W(qA-mK_W)
\cong
\bigoplus_{n\ge 0}\cO_W((n+q)A-mK_W).
\]
Equivalently,
\begin{equation}\label{eq:pushforward-Lq}
H^0(Y,\mathscr L_q)
\cong
\bigoplus_{r\ge q} H^0\bigl(W,\cO_W(rA-mK_W)\bigr).
\end{equation}
If we choose \(q=q_m\), then the right-hand side is exactly the graded module
in~\eqref{eq:canonical-module-cone}.  Therefore
\[
H^0\bigl(Y,\mathscr L_{q_m}\bigr)
\cong
H^0\bigl(C,\cO_C(-mK_C)\bigr)
\]
as graded \(R(W,A)\)-modules.  Since \(C\) is affine, this identifies the
pushforward \(\pi_*\mathscr L_{q_m}\) with \(\cO_C(-mK_C)\).

By construction, \(\mathscr L_{q_m}\) agrees with the pullback of
\(\cO_C(-mK_C)\) on \(U=Y\setminus W_0\).  Since both are rank-one reflexive
sheaves on the smooth variety \(Y\), the reflexive hull of
\(\cO_C(-mK_C)\cdot\cO_Y\) is exactly \(\mathscr L_{q_m}\).  This is the
natural pullback by definition, so
\[
\cO_Y\bigl(-\pi^{\natural}(mK_C)\bigr)=\mathscr L_{q_m}
=p^*\cO_W(-mK_W)(-q_mW_0),
\]
as claimed.
\end{proof}

\begin{corollary}\label{cor:KmYoverC}
For every integer \(m\ge 1\),
\[
K_{m,Y/C}:=K_Y-\frac{1}{m}\pi^{\natural}(mK_C)
= -\left(1+\frac{q_m}{m}\right)W_0.
\]
\end{corollary}

\begin{proof}
By Proposition~\ref{prop:blowup-vertex}(d),
\(K_Y=p^*K_W-W_0\).  By Proposition~\ref{prop:natural-pullback},
\(\pi^{\natural}(mK_C)=p^*(mK_W)+q_mW_0\).  Subtracting gives
\[
K_{m,Y/C}
= p^*K_W-W_0-p^*K_W-\frac{q_m}{m}W_0
= -\left(1+\frac{q_m}{m}\right)W_0.
\qedhere
\]
\end{proof}

\begin{corollary}(cf. \cite[\S3.1]{BdFFU})\label{prop:irrational-discrepancy}
Let \(v_0:=\ord_{W_0}\).  Then for every integer \(m\ge 1\), the
\(m\)-limiting log discrepancy of \(v_0\) in the sense of
\cite[\S3.1]{BdFFU} is
\[
A_C^{(m)}(v_0)=-\frac{q_m}{m}.
\]
Consequently,
\[
A_C(v_0)=\lim_{m\to\infty} A_C^{(m)}(v_0)=-\vartheta(W,A).
\]
In particular, if \(\vartheta(W,A)\) is irrational, then \(C\) has an
irrational log discrepancy.
\end{corollary}

\begin{proof}
Follows directly from Corollary \ref{cor:KmYoverC}.
\end{proof}

Let \(t\ge 0\).  De Fernex--Hacon define \(\cJ_m(C;\m^t)\) by choosing a log
resolution of the pair \((C,\m+\cO_C(-mK_C))\).

\begin{theorem}\label{thm:Jm-explicit}
For every integer \(m\ge 1\) and every real number \(t\ge 0\),
\[
\cJ_m(C;\m^t)
=
\m^{\max\{0,\lfloor t+1+q_m/m\rfloor\}}.
\]
\end{theorem}

\begin{proof}
Choose a log resolution \(g:\widetilde Y\to Y\) of the ideal sheaf
\(\cO_C(-mK_C)\cdot\cO_Y\); then
\(\widetilde\pi:=\pi\circ g:\widetilde Y\to C\) is a log resolution of
\((C,\m+\cO_C(-mK_C))\).  Since \(Y\) is smooth
\[
K_{m,\widetilde Y/C}=K_{\widetilde Y/Y}+g^*K_{m,Y/C}
\]
(see \cite[Lemma~3.5]{dFH}).  By
Corollary~\ref{cor:KmYoverC},
\[
K_{m,\widetilde Y/C}-t\,\widetilde\pi^{-1}(\m)
=K_{\widetilde Y/Y}-\left(t+1+\frac{q_m}{m}\right)g^*W_0,
\]
because \(\m\cO_Y=\cO_Y(-W_0)\) and hence
\(\widetilde\pi^{-1}(\m)=g^*W_0\).

Set
\(\beta_m(t):=t+1+q_m/m\).  Then
\[
\cJ_m(C;\m^t)
=\widetilde\pi_*\cO_{\widetilde Y}\bigl(K_{\widetilde Y/Y}+\lceil-\beta_m(t)\,g^*W_0\rceil\bigr).
\]
Now \(W_0\subset Y\) is a smooth divisor, so 
\cite[Lemma~4.6]{dFH} ( \cite[Lemma 9.2.19, Remark 9.2.10]{Positivity-II}) yields
\[
g_*\cO_{\widetilde Y}\bigl(K_{\widetilde Y/Y}+\lceil-\beta_m(t)\,g^*W_0\rceil\bigr)
=
\cO_Y\bigl(\lceil-\beta_m(t)W_0\rceil\bigr)
=
\cO_Y\bigl(-\lfloor\beta_m(t)\rfloor W_0\bigr).
\]
Pushing forward further by \(\pi\) gives
\[
\cJ_m(C;\m^t)
=
\pi_*\cO_Y\bigl(-\lfloor\beta_m(t)\rfloor W_0\bigr).
\]
If \(\lfloor\beta_m(t)\rfloor\ge 0\), then Proposition~\ref{prop:blowup-vertex}(c)
gives
\(\pi_*\cO_Y(-\lfloor\beta_m(t)\rfloor W_0)=\m^{\lfloor\beta_m(t)\rfloor}\).
If \(\lfloor\beta_m(t)\rfloor<0\), then the same proposition gives
\(\pi_*\cO_Y(-\lfloor\beta_m(t)\rfloor W_0)=\cO_C\).  Combining the two cases
produces exactly the stated formula.
\end{proof}

For a normal variety \(X\) and a formal pair \((X,\mathfrak a^t)\), de Fernex--Hacon defined multiplier ideal
\(\cJ_{\mathrm{dFH}}(X;\mathfrak a^t)\) to be the unique maximal element among
\(\{\cJ_m(X;\mathfrak a^t)\}_{m\ge 1}\) as in
\cite[Proposition~4.7 and Definition~4.8]{dFH}.

I thank my advisor Karl Schwede for asking a question on an earlier draft which resulted the following Theorem \ref{thm:dFH-direct}.

\begin{theorem}(cf. \cite[Theorem 3.6]{Urbinati})\label{thm:Jmax-cone}\label{thm:dFH-direct}
Let \(C\) be the affine cone over \((W,A)\) as above, and let
\(\vartheta(W,A)\) be from
\S~\ref{def:theta}.  Then for every real number \(t\ge 0\),
\[
\cJ_{\mathrm{dFH}}(C;\m^t)
=
\m^{\max\{0,\lfloor t+1+\vartheta(W,A)\rfloor\}}.
\]
Equivalently, if \(K_W\) is pseudo-effective, then
\[
\cJ_{\mathrm{dFH}}(C;\m^t)
=
\m^{\lfloor t+t_0(W,A)+1\rfloor}.
\]
In particular, if $t_0\notin\Q$ then all jumping numbers of $\cJ_{dFH}(C;\m^t)$ are precisely the irrational numbers
\[
n+\kappa \qquad (n\in\Z_{\ge 0}).
\]
\end{theorem}

\begin{proof} By Theorem~\ref{thm:Jm-explicit}, these are
powers of the maximal ideal:
\[
\cJ_m(C;\m^t)=\m^{a_m(t)},
\qquad
 a_m(t):=\max\{0,\lfloor t+1+q_m/m\rfloor\}.
\]
By Lemma~\ref{lem:qm-limit}, \(q_m/m\to\vartheta(W,A)\).  Therefore, for all
sufficiently large \(m\),
\[
\lfloor t+1+q_m/m\rfloor=\lfloor t+1+\vartheta(W,A)\rfloor,
\]
so eventually
\[
a_m(t)=\max\{0,\lfloor t+1+\vartheta(W,A)\rfloor\}.
\]
Set
\[
e(t):=\max\{0,\lfloor t+1+\vartheta(W,A)\rfloor\}.
\]
Then for all sufficiently divisible \(m\) one has \(\cJ_m(C;\m^t)=\m^{e(t)}\).
On the other hand, since each \(q_m/m\ge \vartheta(W,A)\), we have
\(a_m(t)\ge e(t)\) for every \(m\), hence
\(\cJ_m(C;\m^t)\subseteq \m^{e(t)}\) for every \(m\).  Thus \(\m^{e(t)}\) is both
an upper bound for the family \(\{\cJ_m(C;\m^t)\}_m\) and one of its elements.
It is therefore the unique maximal element.

If \(K_W\) is pseudo-effective, then \(\vartheta(W,A)=t_0(W,A)\) by
Lemma~\ref{lem:theta=t0}.  This gives the second formula.
\end{proof}

\begin{proposition}\label{prop:boundary-realization}
Assume that \(W\) is smooth projective over an infinite field and that
\(K_W\) is pseudo-effective, so \(\vartheta(W,A)=t_0(W,A)\).  Fix \(t\ge 0\).
Then there exists an effective \(\Q\)-divisor \(\Gamma\) on \(C\) such that
\(K_C+\Gamma\) is \(\Q\)-Cartier and
\[
\cJ\bigl(C,\Gamma;\m^t\bigr)=\cJ_{\mathrm{dFH}}(C;\m^t).
\]
More concretely, one may take \(\Gamma=\Gamma_s\) from Proposition \ref{prop:Gamma-s} for any rational number
\(s>t_0(W,A)\) sufficiently close to \(t_0(W,A)\)  construction.
\end{proposition}

\begin{proof}
Fix \(t\ge 0\).  Choose a rational number \(s>t_0(W,A)\) so close to
\(t_0(W,A)\) that
\[
\lfloor t+s+1\rfloor=\lfloor t+t_0(W,A)+1\rfloor.
\]
Rest follows by Proposition \ref{prop:J-Gamma-s} and Theorem~\ref{thm:dFH-direct}.
\end{proof}

Now we prove ray-wise agreement $\tau_b=\mathcal J_{dFH}$ away from $\kappa$. The main difference to Lemma \ref{cor:pair-asymptotic-ray} is that this is valid \emph{for any $\alpha\in[0,1)$ with $\alpha\neq \kappa :=1-t_0(W,A)$.}

\begin{proposition}\label{prop:raywise-agreement}
Fix any $\alpha\in[0,1)$ with $\alpha\neq \kappa:=1-t_0(W,A)$.
Then there exists an integer $N(\alpha)$ such that for every integer $n\ge N(\alpha)$,
\[
\tau_b(R;\mathfrak m^{n+\alpha})
=
\mathcal J_{dFH}(R;\mathfrak m^{n+\alpha})
=
\mathfrak m^{\lfloor n+\alpha+t_0+1\rfloor}.
\]
\end{proposition}

\begin{proof}
Set $\delta:=\{\alpha+t_0+1\}\in(0,1)$ \footnote{this is where we use $\alpha\neq \kappa$}.
Choose a rational number $s>t_0$ satisfying
\begin{equation}\label{eq:choose-s-close}
0<s-t_0< 1-\delta,
\end{equation}
and with denominator not divisible by $p$ \footnote{such rationals are dense in $\R$}.
For any integer $n\ge 0$, we have
\[
\lfloor n+\alpha+s+1\rfloor
=
n+\lfloor \alpha+s+1\rfloor,
\qquad
\lfloor n+\alpha+t_0+1\rfloor
=
n+\lfloor \alpha+t_0+1\rfloor.
\]
Since $s>t_0$, we have $\alpha+s+1>\alpha+t_0+1$, hence
$\lfloor \alpha+s+1\rfloor\ge \lfloor \alpha+t_0+1\rfloor$.
On the other hand, using \eqref{eq:choose-s-close} we have
\[
\alpha+s+1
<
\alpha+t_0+1 + (1-\delta)
=
\lfloor \alpha+t_0+1\rfloor+\delta + (1-\delta)
=
\lfloor \alpha+t_0+1\rfloor+1,
\]
so $\lfloor \alpha+s+1\rfloor\le \lfloor \alpha+t_0+1\rfloor$.
Therefore we have equality:
\begin{equation}\label{eq:floor-equality}
\lfloor n+\alpha+s+1\rfloor=\lfloor n+\alpha+t_0+1\rfloor
\qquad\text{for all integers }n\ge 0.
\end{equation}

Now apply Proposition~\ref{prop:tau-Gamma-s}
to the boundary $\Gamma_s$ to get an integer $N(\alpha) = T_s$ such that for all $n\ge N(\alpha,s)$,
\[
\tau\bigl(R;\Gamma_s,\mathfrak m^{n+\alpha}\bigr)
=
\mathfrak m^{\lfloor n+\alpha+s+1\rfloor}.
\]
Using \eqref{eq:floor-equality} and Theorem \ref{thm:Jmax-cone}, this becomes
\[
\tau\bigl(R;\Gamma_s,\mathfrak m^{n+\alpha}\bigr)
=
\mathfrak m^{\lfloor n+\alpha+t_0+1\rfloor}
=
\mathcal J_{dFH}(R;\mathfrak m^{n+\alpha}).
\]
Now using Schwede's sum
\cite[Cor. 5.2]{SchwedeBig}, we get one containment
\[
\tau_b(R;\mathfrak m^{n+\alpha})
\supseteq
\tau\bigl(R;\Gamma_s,\mathfrak m^{n+\alpha}\bigr)
=
\mathcal J_{dFH}(R;\mathfrak m^{n+\alpha}).
\]
The reverse containment
$\tau_b(R;\mathfrak m^{n+\alpha})\subseteq \mathcal J_{dFH}(R;\mathfrak m^{n+\alpha})$
is just Lemma \ref{lem:tau-in-J}.
Thus equality holds for all $n\ge N(\alpha)$.
\end{proof}

\begin{remark}\label{rmk:kappa-ray}
For $\alpha=\kappa$, Proposition~\ref{prop:big-values} already computes $\tau_b(R;\mathfrak m^{n+\kappa})$ yields the same equality
$\tau_b(R;\mathfrak m^{n+\kappa})=\mathcal J_{dFH}(R;\mathfrak m^{n+\kappa})$
for all $n\gg 0$.
\end{remark}

\subsection{What this buys me}

The punchline of this section is the following: if $t_0(W,A)$ is irrational and $s>t_0$ is rational but extremely close to $t_0$,
then the explicit formula $\m^{\floor{s+1+t}}$ forces the ideal to jump at exponents very close to the irrational numbers
\[
n+\kappa,\qquad \kappa:=1-t_0(W,A).
\]
The remaining work is to produce surfaces with irrational $t_0(W,A)$ and then upgrade ``jumps in shrinking intervals'' to actual irrational jumping numbers for $\taub$.
That is the content of the next sections.

\section{Example I: one-parameter family on $E^g, \ \ g \ge 2$}\label{sec:examples-abelian-EE}

The first half of this section produces a smooth surface $W$ and an ample divisor $A$ on $W$, very similar to \cite{Urbinati}, for which $t_0(W,A)$ is irrational. When the parameter $k=2$ below, this specializes to Urbinati's original example \cite{Urbinati}.

Let $E$ be an elliptic curve over a field $k$ \emph{without complex multiplication}.
Let $X=E\times E$.
Let $f_1$ and $f_2$ denote the numerical classes of fibers of the two projections,
and let $\delta$ denote the class of the diagonal.
Then $\NS(X)_\R$ contains the span of $\{f_1,f_2,\delta\}$ with intersection numbers (\cite[Remark 3.4]{BauerBorntraeger18})
\[
f_1^2=f_2^2=\delta^2=0,\qquad
f_1\cdot f_2=f_1\cdot \delta=f_2\cdot \delta=1.
\]
In particular, for any class $D=af_1+bf_2+c\delta$ one has
\[
D^2 = 2(ab+ac+bc).
\]
I will follow \cite{BirkenhakeLange,Rosen-Shnidman} closely in this section.
\begin{lemma}(\cite{Bauer-AbVar,BirkenhakeLange,LazPosI}.)\label{lem:psef-nef-abelian}
Let $X$ be an abelian surface over any field.
Then every effective divisor on $X$ is nef.
Consequently, the pseudo-effective cone and the nef cone coincide: $\psef(X)=\nef(X)$ and its boundary inside the positive cone is cut out by the isotropic condition $D^2=0$.

\end{lemma}

Since $K_X\sim 0$, the threshold computations reduce to analyzing when a divisor becomes nef, equivalently when it lies in the closure of the positive cone.

\begin{lem}(\cite{Bauer-Schulz,BauerBorntraeger18})\label{lem:nef-cone-exex}
Assume $\operatorname{End}(E)=\Z$.
Let $X:=E\times E$ and write a real divisor class as $D=af_1+bf_2+c\delta\in NS(X)_\R$.
Under the identification of $NS(X)_\R$ with the space of symmetric $2\times 2$ real matrices given by
\[
D\longmapsto
H(D):=
\begin{pmatrix}
a+c & c\\
c & b+c
\end{pmatrix},
\]
one has $D^2=2\det(H(D))$.
Moreover, $D$ is nef if and only if $H(D)$ is positive semidefinite, equivalently
\[
a+c\ge 0,\qquad b+c\ge 0,\qquad ab+ac+bc\ge 0.
\]
\end{lem}

\subsection{A one-parameter family}

Fix an integer $k\ge 2$ and consider divisors
\[
L_k := 3(2f_1+f_2+k\delta),
\qquad
M := 3(f_1+f_2).
\]
We compute the pseudo-effective threshold
\[
t_0(X;L_k,M) := \inf\{t\ge 0 \mid tL_k - M \ \text{is pseudo-effective}\}.
\]
Because $\psef=\nef$ and $X$ is a surface, a nef divisor on the boundary of the nef cone has self-intersection zero.
Thus $t_0$ is the larger root of $(tL_k-M)^2=0$.

\begin{proposition}\label{prop:t0-EE}
With notation as above,
\[
L_k^2 = 18(3k+2),\qquad
L_k\cdot M = 9(2k+3),\qquad
M^2=18,
\]
and the threshold is
\[
t_0(X;L_k,M) = \frac{(2k+3)+\sqrt{4k^2+1}}{2(3k+2)}.
\]
In particular, $t_0(X;L_k,M)$ is irrational for every $k\ge 2$.
\end{proposition}

\begin{proof}
For $t\in\R$ write $D_t:=tL_k-M$.
Using the intersection table, we compute the intersections with the nef boundary classes:
\begin{align*}
D_t\cdot f_1 &= t(L_k\cdot f_1)-(M\cdot f_1)
= t(3+3k)-3
=3\big(t(1+k)-1\big),\\
D_t\cdot f_2 &= t(L_k\cdot f_2)-(M\cdot f_2)
= t(6+3k)-3
=3\big(t(2+k)-1\big),\\
D_t\cdot \delta &= t(L_k\cdot \delta)-(M\cdot \delta)
=9t-6
=3(3t-2).
\end{align*}
Thus $D_t\cdot f_1\ge 0$ for $t\ge 1/(1+k)$, $D_t\cdot f_2\ge 0$ for $t\ge 1/(2+k)$, and
$D_t\cdot\delta\ge 0$ for $t\ge 2/3$.
Since the threshold $t_0$ computed below is always $>2/3$, these inequalities are automatic at~$t=t_0$,
and nefness is detected by the remaining condition $D_t^2\ge 0$ from Lemma~\ref{lem:nef-cone-exex}.

We now compute $D_t^2$.
First,
\[
L_k^2=(6f_1+3f_2+3k\delta)^2
=2\big(18+18k+9k\big)=36+54k=18(3k+2),
\]
\[
L_k\cdot M=(6f_1+3f_2+3k\delta)\cdot(3f_1+3f_2)
=18+9+9k+9k=27+18k=9(2k+3),
\]
and $M^2=18$.
Therefore
\[
D_t^2=(tL_k-M)^2
=t^2L_k^2-2t(L_k\cdot M)+M^2
=18\Big((3k+2)t^2-(2k+3)t+1\Big).
\]
The quadratic in parentheses has discriminant
\[
(2k+3)^2-4(3k+2)=4k^2+1.
\]
For $k\ge 1$ this is not a perfect square: if $m^2=4k^2+1$, then
\[
(m-2k)(m+2k)=1
\]
forces $m=1$ and $k=0$, a contradiction.
Hence the larger real root is irrational and equals
\[
t_0=\frac{(2k+3)+\sqrt{4k^2+1}}{2(3k+2)}=\frac{(2k+3)+\sqrt{4k^2+1}}{6k+4}.
\]
Since $D_t^2<0$ for $t<t_0$ and $D_{t_0}^2=0$, Lemma~\ref{lem:nef-cone-exex} yields that
$D_t$ is nef if and only if $t\ge t_0$.
Since $4k^2+1$ is not a square for $k\ge 1$, the root is irrational.
\end{proof}

\begin{remark}
For $k=2$ this gives
\[
t_0=\frac{7+\sqrt{17}}{16},
\qquad
\kappa=1-t_0=\frac{9-\sqrt{17}}{16},
\]
which matches the quadratic irrational appearing in \cite{Urbinati}.
\end{remark}

We now modify the surface by taking a double cover branched in $|2M|$ so that $K_W$ is numerically nontrivial and aligned with $M$.

\begin{proposition}(\cite{BirkenhakeLange,Rosen-Shnidman})\label{prop:double-cover}
Let $X$ be a smooth projective variety over an infinite field $k$ with $\mathrm{char}(k)\ne 2$.
Let $M$ be an ample divisor on $X$ such that $|2M|$ contains a smooth member $B$.
Let $f:W\to X$ be the double cover branched along $B$.
Then $W$ is smooth and
\[
K_W \sim f^*(K_X+M).
\]
\end{proposition}

\begin{proof}
This is standard, see \cite{Pardini91,BirkenhakeLange,Rosen-Shnidman}.
\end{proof}

We apply this with $X=E\times E$ (so $K_X\sim 0$) and the divisor $M=3(f_1+f_2)$, which is ample and $|2M|$ is globally generated, see \cite[Theorem (1)]{PP}.
Over an infinite field, Bertini guarantees a smooth $B\in|2M|$. Set $A:=f^*L_k$ on $W$.
Then $A$ is ample and $K_W\sim f^*M$.

\begin{lemma}(cf. \cite{Urbinati})\label{lem:finite-threshold}
Let \(f:W\to X\) be a finite surjective morphism of smooth projective varieties
of arbitrary dimension.  Let \(L\) and \(M\) be Cartier divisors on \(X\), and
set \(A:=f^*L\), \(B:=f^*M\).  Then
\[
\inf\{t\in\R\mid tA-B\in\overline{\Eff}(W)\}
=
\inf\{t\in\R\mid tL-M\in\overline{\Eff}(X)\}.
\]
In particular, pseudo-effective threshold $\vartheta(W,A)$ remains invariant under finite cover.
\end{lemma}

\begin{proof}
If \(tL-M\in\overline{\Eff}(X)\), then pulling back gives
\(tA-B=f^*(tL-M)\in\overline{\Eff}(W)\).

Conversely suppose that \(tA-B\in\overline{\Eff}(W)\).  Push forward by \(f\).
Since the pushforward of an effective divisor under a finite surjective map is
effective, we get
\[
f_*(tA-B)=f_*f^*(tL-M)=(\deg f)(tL-M)\in\overline{\Eff}(X).
\]
This implies
\(tL-M\in\overline{\Eff}(X)\).  The two threshold sets therefore coincide.
\end{proof}

\begin{corollary}\label{cor:W-EE}
Let $W\to E\times E$ be the double cover branched along a smooth divisor in $|2M|$, where $M=3(f_1+f_2)$.
Let $A=f^*L_k$ with $L_k=3(2f_1+f_2+k\delta)$.
Then
\[
t_0(W,A)=t_0(X;L_k,M)=\frac{(2k+3)+\sqrt{4k^2+1}}{2(3k+2)}\notin\Q.
\]
\end{corollary}

We now have a supply of smooth surfaces $(W,A)$ with irrational $t_0$.
In later sections we will input these into the cone constructions.

\begin{remark}
    Following \cite[\S4.1]{Bauer-Schulz}, one can do similar computations when $E$ has \emph{complex multiplication} as well. If $\End(E)=\mathbb Z[\omega]$ with $\omega = e^{2\pi i/3}$ and $\Sigma$ is the graph of multiplication by $\omega$ then (see \cite[\S4.1]{Bauer-Schulz}) we have
\[
\NS(X)\cong \mathbb Z F_1 \oplus \mathbb Z F_2 \oplus \mathbb Z \Delta \oplus \mathbb Z \Sigma,
\]
Picking $M := F_1+F_2$ and $L := F_1+F_2+\Delta+\Sigma$ gives $t_0(X;L,M)=\dfrac{3+\sqrt3}{6}$.
\end{remark}

\subsection{Higher dimensional examples}\label{high-dim-example}

The key point is that on a product of elliptic curves,
the nef cone is described by \emph{positive semidefinite symmetric matrices}, so
its boundary is cut out by a determinant. We follow \cite{BirkenhakeLange,Pren-Smith} closely.

\begin{theorem}(\cite[Ch. 5]{BirkenhakeLange}, \cite[\S1--\S3]{Pren-Smith})
\label{thm:NS-Eg}
Let $E$ be an elliptic curve over an algebraically closed field of characteristic $0$
with $\End(E)=\mathbb{Z}$, and let $X := E^g$.
Fix the product polarization $\Theta$ on $X$.
Then there is a canonical identification
\[
\NS(X)_{\mathbb{R}} \;\cong\; Sym_g(\mathbb{R}),
\]
under which the nef cone $\Nef(X)$ corresponds to the cone of positive semidefinite
symmetric matrices, and the ample cone corresponds to the positive definite ones. Moreover, on an abelian variety a divisor class is effective if and only if it is nef
\cite[Proposition~1.2]{Pren-Smith}, hence
\[
\overline{Eff}(X) \;=\; \Nef(X)
\qquad\text{inside } N^1(X)_{\R}.
\]
\end{theorem}

\begin{corollary}\label{lem:NS-Eg-thrsh}(\cite{BirkenhakeLange,Pren-Smith})
In the setup of the above Theorem \ref{thm:NS-Eg}, let \(L\) and \(M\) be ample divisor classes corresponding to positive definite
symmetric matrices \(P\) and \(Q\).  Then
\[
\vartheta(X;L,M):=\inf\{t\in\R\mid tL-M\in\overline{\Eff}(X)\}
\]
is the largest generalized eigenvalue of the pair \((Q,P)\), namely
\[
\vartheta(X;L,M)=\lambda_{\max}(P^{-1/2}QP^{-1/2}).
\]
In particular, if \(P=nI_g\), then
\[
\vartheta(X;L,M)=\frac{\lambda_{\max}(Q)}{n}.
\]
\end{corollary}

\begin{definition}
For \(n\ge 0\), define polynomials\footnote{These are rescaled Chebyshev polynomials of the second kind, see \url{https://en.wikipedia.org/wiki/Chebyshev_polynomials}. They naturally appear as characteristic polynomials of tridiagonal matrices.} \(R_n(x)\in \Z[x]\) recursively by
\[
R_0(x)=1,
\qquad
R_1(x)=x,
\qquad
R_n(x)=xR_{n-1}(x)-R_{n-2}(x)\quad (n\ge 2).
\]
\end{definition}

\begin{lemma}\label{lem:tridiagonal-charpoly}
For \(g\ge 2\) and a parameter \(a\), let
\[
T_g(a)=
\begin{pmatrix}
0&1&0&\cdots&0\\
1&0&1&\ddots&\vdots\\
0&1&0&\ddots&0\\
\vdots&\ddots&\ddots&0&1\\
0&\cdots&0&1&a
\end{pmatrix}.
\]
Then the characteristic polynomial of \(T_g(a)\) is
\[
\chi_{T_g(a)}(x)=R_g(x)-aR_{g-1}(x).
\]
\end{lemma}

\begin{proof}
Let \(D_n(x)\) denote the determinant of the \(n\times n\) tridiagonal matrix
with diagonal entries \(x,\dots,x\) and with \(-1\) on the super- and
sub-diagonals.  Expanding along the last row gives the familiar recurrence
\[
D_n(x)=xD_{n-1}(x)-D_{n-2}(x),
\qquad
D_0(x)=1,
\quad
D_1(x)=x.
\]
Hence \(D_n(x)=R_n(x)\) for all \(n\).

Now compute
\(\chi_{T_g(a)}(x)=\det(xI_g-T_g(a))\).  The matrix \(xI_g-T_g(a)\) has diagonal
entries \(x,\dots,x,x-a\) and \(-1\) on the adjacent off-diagonals.  Expanding
again along the last row yields
\[
\chi_{T_g(a)}(x)=(x-a)R_{g-1}(x)-R_{g-2}(x).
\]
Using the recurrence relation for \(R_g\), namely
\(R_g(x)=xR_{g-1}(x)-R_{g-2}(x)\), we obtain
\[
\chi_{T_g(a)}(x)=R_g(x)-aR_{g-1}(x),
\]
as claimed.
\end{proof}

\begin{lemma}\label{lem:coprime-Rn}
For every \(n\ge 1\), the polynomials \(R_n\) and \(R_{n-1}\) are coprime in
\(\Q[x]\).
\end{lemma}

\begin{proof}
If a polynomial \(G\in \Q[x]\) divides both \(R_n\) and \(R_{n-1}\), then by the
recurrence it also divides
\[
R_{n-2}=xR_{n-1}-R_n.
\]
Descending inductively, \(G\) divides \(R_1=x\) and \(R_0=1\), hence \(G\) is a
unit.  Therefore \(\gcd(R_n,R_{n-1})=1\).
\end{proof}

\begin{proposition}\label{prop:irreducible-family}
For every \(g\ge 2\), the polynomial
\[
F_g(a,x):=R_g(x)-aR_{g-1}(x)
\in \Q[a,x]
\]
is irreducible as a polynomial in \(x\) over \(\Q(a)\).
\end{proposition}

\begin{proof}
We first view \(F_g(a,x)\) as a polynomial in the variable \(a\) with
coefficients in the UFD \(\Q[x]\):
\[
F_g(a,x)=R_g(x)-aR_{g-1}(x).
\]
By Lemma~\ref{lem:coprime-Rn}, the coefficients \(R_g(x)\) and
\(R_{g-1}(x)\) are coprime in \(\Q[x]\).  Therefore \(F_g(a,x)\) is primitive in
\(\Q[x][a]\), and because it has degree~\(1\) in \(a\), it is irreducible in the
polynomial ring \(\Q[x][a]\).

Now suppose, for contradiction, that \(F_g(a,x)\) factors in \(\Q(a)[x]\).
Clearing denominators in \(a\), one obtains a nontrivial factorization in
\(\Q[a,x]\).  By Gauss' lemma this would give a nontrivial factorization in the
UFD \(\Q[x][a]\), contradicting the previous paragraph.  Hence
\(F_g(a,x)\) is irreducible in \(\Q(a)[x]\).
\end{proof}

\begin{theorem}\label{thm:appendix-any-degree-proof}
Fix an integer \(g\ge 2\).  There exist positive definite symmetric matrices
\(P,Q\in \operatorname{Sym}_g(\Q)\) such that the generalized eigenvalue
polynomial
\[
\det(Q-\lambda P)
\]
is irreducible of degree \(g\) over \(\Q\).  Consequently, there is an ample
pair \((L,M)\) on \(E^g\) whose threshold
\[
\inf\{t\in \R\mid tL-M\in \overline{\Eff}(E^g)\}
\]
has exact algebraic degree \(g\).
\end{theorem}

\begin{proof}
Let \(F_g(a,x)\) be the irreducible family from
Proposition~\ref{prop:irreducible-family}.  By Hilbert irreducibility, there exist infinitely many integers \(a_0\) such
that the specialization
\[
F_g(a_0,x)=R_g(x)-a_0R_{g-1}(x)
\]
remains irreducible in \(\Q[x]\).  Fix such an \(a_0\), and let
\(\lambda_1,\dots,\lambda_g\in \R\) be the eigenvalues of the real symmetric
matrix \(T_g(a_0)\).  Because the characteristic polynomial is irreducible of
degree \(g\), the largest eigenvalue \(\lambda_{\max}\) has exact algebraic degree
\(g\).

Choose an integer \(c> -\lambda_g\), where \(\lambda_g\) is the smallest
eigenvalue of \(T_g(a_0)\).  Then
\[
Q:=T_g(a_0)+cI_g
\]
is positive definite.  Its characteristic polynomial is
\[
\det(xI_g-Q)=F_g(a_0,x-c),
\]
which is still irreducible over \(\Q\), so the largest eigenvalue of \(Q\) also
has exact degree \(g\).

Now choose an integer \(n\) strictly larger than the largest eigenvalue of
\(Q\), and set \(P=nI_g\).  Then \(P\) is positive definite and
\[
\det(Q-\lambda P)=n^g\det\!\left(\frac{1}{n}Q-\lambda I_g\right).
\]
Hence the generalized eigenvalues of the pair \((Q,P)\) are exactly the numbers
\(\mu_i/n\), where \(\mu_i\) are the eigenvalues of \(Q\).  In particular, the
largest generalized eigenvalue is \(\lambda_{\max}(Q)/n\), again of degree
\(g\).

Finally, applying Theorem \ref{thm:NS-Eg} and Corollary \ref{lem:NS-Eg-thrsh}, we get
\[
\tau=\lambda_{\max}(P^{-1/2}QP^{-1/2})=\frac{\lambda_{\max}(Q)}{n}.
\]
This number has algebraic degree exactly \(g\), as required.
\end{proof}

\begin{remark}\label{rem:appendix-any-degree-cone}
To get irrational jumping numbers from Theorem~\ref{thm:appendix-any-degree-proof}, choose a
smooth double cover \(f\colon W\to E^g\) branched in a smooth divisor from
\(|2M|\), set \(A=f^*L\), and invoke Lemma~\ref{lem:finite-threshold}.  Then
\(K_W=f^*M\), so \(\vartheta(W,A)\) equals \(\tau\), and Theorem~\ref{thm:dFH-direct} gives irrational jumping
numbers on the cone over \((W,A)\).
\end{remark}

\section{Example II: Products of isogenous elliptic curves}\label{sec:examples-isogenous}

This family produces a different quadratic irrational depending on the degree of an isogeny. I will follow \cite{BauerBorntraeger18,BirkenhakeLange,Rosen-Shnidman} closely in this section.

Let $E_1,E_2$ be isogenous elliptic curves withoutcomplex multiplication over an infinite field $k$ and assume there is a minimal isogeny $\varphi:E_1\to E_2$ of degree $d\ge 1$.
Let $X=E_1\times E_2$.
Let $F_1$ and $F_2$ be the fiber classes of the projections, and let $\Gamma$ be the class of the graph of $\varphi$.
Then, by \cite[\S 3, Type B]{BauerBorntraeger18}\footnote{\cite{BauerBorntraeger18}, uses $F_3 = \Gamma - d F_1 -F_2$ as their 3rd basis element.} (or \cite{Rosen-Shnidman})
\[
F_1^2=F_2^2=\Gamma^2=0,\qquad
F_1\cdot F_2=1,\qquad
F_1\cdot \Gamma=1,\qquad
F_2\cdot \Gamma=d.
\]

As in Lemma~\ref{lem:psef-nef-abelian}, $\psef(X)=\nef(X)$, see \cite{Bauer-AbVar, BirkenhakeLange, Rosen-Shnidman}.

\begin{lem}(\cite{BirkenhakeLange,BauerBorntraeger18,Rosen-Shnidman})\label{lem:nef-cone-isog}
Write $D=xF_1+yF_2+z\Gamma\in NS(X)_\R$.
Set
\[
H(D):=
\begin{pmatrix}
x+dz & \sqrt{d}\,z\\
\sqrt{d}\,z & y+z
\end{pmatrix}\in \operatorname{Sym}_2(\R).
\]
Then $D^2=2\det(H(D))$.
Moreover, $D$ is nef if and only if $H(D)$ is positive semidefinite, equivalently
\[
x+dz\ge 0,\qquad y+z\ge 0,\qquad xy+xz+dyz\ge 0.
\]
\end{lem}

\subsection{Irrational threshold}

Define divisors
\[
L := F_1+2F_2+\Gamma,
\qquad
M := F_1+F_2.
\]
Then $L$ and $M$ are ample for $d\ge 1$.

\begin{proposition}\label{prop:t0-isogenous}
With notation above,
\[
L^2 = 2(2d+3),\qquad L\cdot M = d+4,\qquad M^2=2,
\]
and the pseudo-effective threshold is
\[
t_0(X;L,M) = \frac{(d+4)+\sqrt{d^2+4}}{2(2d+3)}.
\]
In particular, $t_0(X;L,M)$ is irrational for every $d\ge 1$.
\end{proposition}

\begin{proof}
Set $D_t:=tL-M$. We first check intersections with the nef boundary classes:
\begin{align*}
D_t\cdot F_1 &= t(L\cdot F_1)-(M\cdot F_1)
=t(2+1)-1=3t-1,\\
D_t\cdot F_2 &= t(L\cdot F_2)-(M\cdot F_2)
=t(1+d)-1=(d+1)t-1,\\
D_t\cdot \Gamma &= t(L\cdot \Gamma)-(M\cdot \Gamma)
=t(1+2d)- (1+d)=(2d+1)t-(d+1).
\end{align*}
The threshold computed below satisfies $t_0>\max\{1/3,\,1/(d+1),\, (d+1)/(2d+1)\}$,
so these inequalities are automatic at~$t=t_0$.
Thus nefness is again controlled by the single condition $D_t^2\ge 0$.

Next, compute
\[
L^2=2(2d+3),\qquad L\cdot M=d+4,\qquad M^2=2,
\]
hence
\[
D_t^2=t^2L^2-2t(L\cdot M)+M^2
=2\Big((2d+3)t^2-(d+4)t+1\Big).
\]
The discriminant equals
\[
(d+4)^2-4(2d+3)=d^2+4.
\]
For $d\ge 1$ this is not a square: if $m^2=d^2+4$, then
\[
(m-d)(m+d)=4,
\]
forcing $d=0$.
Therefore the larger root
\[
t_0=\frac{(d+4)+\sqrt{d^2+4}}{2(2d+3)}
\]
is irrational, and $D_t$ is nef if and only if $t\ge t_0$.
\end{proof}

\subsection{Passing to a double cover}

As in Section~\ref{sec:examples-abelian-EE}, take a double cover $f:W\to X$ branched along a smooth divisor $B\in |2M|$
(over $\mathrm{char}\ne 2$), so $K_W=f^*M$ by Proposition~\ref{prop:double-cover}.
Let $A:=f^*L$.
Then Lemma~\ref{lem:finite-threshold} implies $t_0(W,A)=t_0(X;L,M)$ is irrational.

\section{Example III: Abelian surfaces with real multiplication}\label{sec:examples-real-mult}

This family is conceptually different than the previous sections: the irrationality comes from real multiplication and the shape of the nef cone. My main reference for this section is \cite{BauerBorntraeger18}.

\subsection{Existence and the N\'eron--Severi lattice(\cite[\S2, Type 1]{BauerBorntraeger18})}

Fix a squarefree integer $d>1$ such that $d \equiv 2, 3 \ (mod \ 4) $ .
There exist \footnote{see \cite{BirkenhakeLange} and \cite{vdGeer88}} simple abelian surfaces $X$ with real multiplication by $\Q(\sqrt{d})$,
i.e.\ $\End_\Q(X)\cong \Q(\sqrt{d})$.
For such surfaces with $\rho(X)=2$, there exist classes $L_0,L_\omega\in\NS(X)$ such that
\[
\NS(X)\simeq \Z L_0\oplus \Z L_\omega,\qquad \mathrm{with \ Gram \ matrix \ }
\begin{pmatrix}L_0^2 & L_0\cdot L_\omega\\ L_\omega\cdot L_0 & L_\omega^2\end{pmatrix}
=
\begin{pmatrix}2&0\\0&-2d\end{pmatrix},
\]
and $L_0$ is ample. This is a special case of \cite[\S 2]{BauerBorntraeger18} (Type~1 with $f=1$).

We will use only the numerical consequence:
\[
L_0^2=2,\qquad L_\omega^2=-2d,\qquad L_0\cdot L_\omega=0.
\]
As before, $\psef(X)=\nef(X)$ by Lemma~\ref{lem:psef-nef-abelian}.

\subsection{Irrational threshold}

Fix an integer $n>\sqrt{d}$ and define
\[
L_n := nL_0 + L_\omega,\qquad M:=L_0.
\]
Then $L_n$ is ample (since it has positive square and intersects $L_0$ positively).

\begin{proposition}\label{prop:t0-real-mult}
With notation above,
\[
L_n^2 = 2(n^2-d),\qquad L_n\cdot M = 2n,\qquad M^2=2,
\]
and the pseudo-effective threshold is
\[
t_0(X;L_n,M) = \frac{n+\sqrt{d}}{n^2-d}=\frac{1}{n-\sqrt{d}}.
\]
In particular, $t_0(X;L_n,M)$ is irrational.
\end{proposition}

\begin{proof}
Compute:
\[
L_n^2 = (nL_0+L_\omega)^2 = n^2L_0^2 + 2n(L_0\cdot L_\omega)+L_\omega^2 = 2n^2-2d =2(n^2-d).
\]
Also $L_n\cdot M = (nL_0+L_\omega)\cdot L_0 = nL_0^2=2n$ and $M^2=2$.

Thus
\[
(tL_n-M)^2 = t^2L_n^2 -2t(L_n\cdot M)+M^2
=2\Bigl((n^2-d)t^2-2nt+1\Bigr).
\]
The larger root is
\[
t_0=\frac{2n+\sqrt{4n^2-4(n^2-d)}}{2(n^2-d)}
=\frac{2n+2\sqrt{d}}{2(n^2-d)}=\frac{n+\sqrt{d}}{n^2-d}=\frac{1}{n-\sqrt{d}}.
\]
\end{proof}

\subsection{Double covers}

As before, take a double cover $f:W\to X$ branched along a smooth divisor in $|2M|$ (over $\mathrm{char}\ne 2$),
so $K_W=f^*M$ and $A=f^*L_n$.
Then $t_0(W,A)=t_0(X;L_n,M)$ is irrational by Lemma~\ref{lem:finite-threshold}.

\begin{remark}
Following \cite[\S2, Type~1]{BauerBorntraeger18}, one can do similar computations when $d\equiv 1\pmod 4$ as well. By \cite[\S2, Type~1]{BauerBorntraeger18} there exists a $\Z$-basis
$\{L_0,L_\infty\}$ of $\NS(X)$ with intersection matrix
\[
\begin{pmatrix}
L_0^2 & L_0\cdot L_\infty\\
L_0\cdot L_\infty & L_\infty^2
\end{pmatrix}
=
\begin{pmatrix}
2 & 1\\
1 & \frac{1-d}{2}
\end{pmatrix}.
\]
Fix an integer $n\gg 0$ and pick $M:=L_0, \ L:=nL_0+L_\infty.$
Then one gets
\[
t_0=\frac{2(2n+1+\sqrt d)}{4n^2+4n+1-d}.
\]
If $d$ is not a square, then $t_0\notin \Q$.
    
\end{remark}

\section{Example IV: Rank-two lattices on $K3$ surfaces}\label{sec:examples-k3}

We now give two $K3$-surface families. The main references which I will follow closely are \cite{Huybrechts16,Kovacs94,Kovacs14}.
The first family reuses the same rank-two lattice as in the real multiplication Section \ref{sec:examples-real-mult}, showing that the mechanism is not special to abelian surfaces.
The second family uses an even lattice with Gram matrix $\begin{psmallmatrix}4&2m\\2m&4\end{psmallmatrix}$, producing a different irrationality.

\subsection{K3 surfaces with N\'eron--Severi lattice $\langle 2\rangle \oplus \langle -2d\rangle$}

Fix a squarefree integer $d>1$ with $d\equiv 3\pmod 4$.
Take a $K3$ surfaces with Picard rank $2$ and N\'eron--Severi lattice
\[
N_d := \langle 2\rangle \oplus \langle -2d\rangle= \Z H\oplus\Z F,
\]

Following \cite{Huybrechts16}, let $H,F$ be a basis of $N_d$ with
\[
H^2=2,\qquad F^2=-2d,\qquad H\cdot F=0.
\]
Choose an integer $u>\sqrt{d}$ and define
\[
L := uH + F,\qquad M:=H.
\]
We compute the pseudo-effective threshold explicitly.

\begin{lemma}\label{lem:t0-Nd}
Assume $\NS(X)\cong N_d$ and that the nef cone equals the positive cone.
Then $L$ and $M$ are ample and
\[
t_0(X;L,M)=\frac{1}{u-\sqrt{d}}\notin\Q.
\]
\end{lemma}

\begin{proof}
Since $M=H$ has $M^2=2>0$, it lies in the positive cone.
The class $L=uH+F$ has
\[
L^2 = 2u^2-2d = 2(u^2-d)>0\qquad\text{and}\qquad L\cdot M = 2u>0,
\]
so $L$ lies in the same positive-cone component as $M$; hence $L$ and $M$ are ample.

Set $D_t:=tL-M$.
Using $H^2=2$, $F^2=-2d$, and $H\cdot F=0$ we compute
\[
D_t = (tu-1)H + tF,\qquad
D_t^2 = 2(tu-1)^2 - 2dt^2.
\]
Since $\nef(X)$ is the closure of the positive cone (by Lemma \ref{lem:k3-nef} below), $D_t$ is pseudo-effective iff $D_t^2\ge 0$ and $D_t\cdot M\ge 0$.
For $t>0$ we have $D_t\cdot M = 2(tu-1)$, so the threshold occurs when $D_t^2=0$ with $tu-1>0$.
Solving $D_t^2=0$ gives
\[(tu-1)^2=dt^2,\qquad tu-1=t\sqrt d,\qquad t=\frac{1}{u-\sqrt d}.
\]
This is irrational since $d$ is not a square.
\end{proof}

It remains to justify the nef-cone hypothesis.
In our lattice, this is automatic:

\begin{lemma}\label{lem:Nd-no}
Let $d>1$ be squarefree with $d\equiv 3\pmod 4$.
Then the lattice $N_d=\langle 2\rangle\oplus\langle -2d\rangle$ has no nonzero vector of square $0$ and no vector of square $-2$.
\end{lemma}

\begin{proof}
Write $v=aH+bF$.
Then
\[v^2 = 2(a^2-db^2).
\]
If $v^2=0$, then $a^2=db^2$.
Since $d$ is not a square, this forces $b=0$ and then $a=0$.
If $v^2=-2$, then $a^2-db^2=-1$.
Reducing modulo $4$ and using $d\equiv 3\pmod 4$, we obtain
\[a^2 - 3b^2 \equiv -1 \pmod 4\qquad\Longleftrightarrow\qquad a^2 + b^2 \equiv 3\pmod 4.
\]
But $a^2$ and $b^2$ are each $0$ or $1$ modulo $4$, so the sum cannot be $3$.
\end{proof}

Combining Lemmas~\ref{lem:Nd-no} and \ref{lem:k3-nef}, we may and do choose $X$ with $\NS(X)\cong N_d$ so that the nef cone equals the closure of the positive cone,
and hence Lemma~\ref{lem:t0-Nd} applies.

\begin{lemma}(\cite{Kovacs94,Kovacs14,Huybrechts16})\label{lem:k3-nef}
Let $X$ be a projective $K3$ surface over an algebraically closed field.
If $\NS(X)$ contains no classes of self-intersection $0$ and no classes of self-intersection $-2$, and, if $\rho(X)=2$, then $\psef(X)=\nef(X)$ in $N^1(X)_\R$ and the two boundary rays are irrational.
\end{lemma}
\begin{proof}
By \cite[Lemma~2.1]{Kovacs14}, for any $v\in \mathrm{NS}(X)$ with $v^2\ge -2$ one of $v$ or $-v$ is effective. In Picard rank $2$, the cone of curves $\overline{\mathrm{NE}}(X)$ has two extremal rays, and \cite[Corollary~3.2 and Remark~1.2]{Kovacs14} shows that either a boundary ray contains an effective class of square $0$ or $-2$, or else neither boundary ray contains an effective class (in which
case the boundary rays are irrational). Under our hypothesis the latter occurs, hence the boundary rays are irrational.

Since $X$ has no $(-2)$-classes, there are no smooth rational curves. Therefore the nef cone has no walls, so $\mathrm{Nef}(X)$ equals the closure of the positive cone $\overline{V^+}$, and consequently $\overline{\mathrm{Eff}}(X)=\overline{V^+}$ as well.
\end{proof}

Under these hypotheses, threshold computations proceed as on an abelian surface: the boundary of the nef cone is given by square-zero classes, giving the quadratic irrational root.

\subsection{A second $K3$ family with Gram matrix $((4,2m),(2m,4))$}

Fix an integer $m\ge 3$ and consider the rank-two even lattice
\[
\Lambda_m := \begin{pmatrix}4&2m\\2m&4\end{pmatrix}.
\]
There exist projective $K3$ surfaces with $\NS(X)\cong \Lambda_m$ and $\rho(X)=2$.

Let $e_1,e_2$ be a basis with
\[
e_1^2=e_2^2=4,\qquad e_1\cdot e_2=2m.
\]
Define
\[
M:=e_1,\qquad L:=e_1+e_2.
\]
Then $M$ and $L$ are ample if we choose the positive cone component appropriately.

\begin{lemma}\label{lem:Lambda-no}
Let $\Lambda_m$ be the rank-two lattice with Gram matrix
\[
\begin{pmatrix}
4 & 2m\\
2m & 4
\end{pmatrix},
\qquad m\ge 3.
\]
Then $\Lambda_m$ contains no nonzero classes of self-intersection $0$ and no classes of self-intersection $-2$.
\end{lemma}

\begin{proof}
Let $v=a e_1 + b e_2 \in \Lambda_m$. Then
\[
v^2 = 4a^2 + 4b^2 + 4mab = 4(a^2+b^2+mab),
\]
so $v^2=-2$ is impossible.

If $v^2=0$, then $a^2+b^2+mab=0$.
If $b=0$, then $a=0$. Otherwise, set $x=a/b\in \Q$ and divide by $b^2$ to obtain
\[
x^2 + mx + 1 = 0.
\]
Thus the discriminant $m^2-4$ is a square in $\Q$, hence a square in $\Z$.
Write $m^2-4=r^2$ with $r\in\Z_{\ge 0}$; then $(m-r)(m+r)=4$, forcing $m=2$.
This contradicts $m\ge 3$, so no nonzero isotropic class exists.
\end{proof}

Therefore Lemma~\ref{lem:k3-nef} applies and $\psef=\nef$ on such $X$.

\begin{proposition}\label{prop:t0-Lambda}
On a $K3$ surface $X$ with $\NS(X)\cong \Lambda_m$ for $m\ge 3$ and with ample cone equal to the positive cone,
the pseudo-effective threshold for $(L,M)$ is
\[
t_0(X;L,M)=\frac{(m+2)+\sqrt{m^2-4}}{2(m+2)}\notin\Q.
\]
\end{proposition}

\begin{proof}
Compute $L^2=(e_1+e_2)^2 =4+4+4m=4(m+2)$,
$L\cdot M=(e_1+e_2)\cdot e_1=4+2m=2(m+2)$,
and $M^2=4$.

\[
(tL-M)^2 = t^2L^2 -2t(L\cdot M)+M^2
=4\Bigl((m+2)t^2-(m+2)t+1\Bigr).
\]
The larger root is
\[
t_0=\frac{(m+2)+\sqrt{(m+2)^2-4(m+2)}}{2(m+2)}
=\frac{(m+2)+\sqrt{m^2-4}}{2(m+2)}.
\]
Irrationality follows from Lemma~\ref{lem:Lambda-no}.
\end{proof}

As before, take a double cover branched in $|2M|$ (for $\mathrm{char}\ne 2$) to obtain a smooth surface $W$ with $K_W=f^*M$
and $A=f^*L$, preserving the threshold.

\subsection{Existence}
\emph{Existence in characteristic $0$.}
Let $L$ be an even lattice of signature $(1,\rho-1)$ with $\rho\le 10$.
By Morrison's theorem \cite{Mor84} (see \cite[Corollary.~14.3.1 and Remark.~14.3.7]{Huybrechts16}),
there exists a complex projective K3 surface $X/\C$ with $\NS(X)\simeq L$.
In particular, for $d>1$ squarefree, the lattices
\[
N_d=\langle 2\rangle \oplus \langle -2d\rangle,\qquad
\Lambda_m=\begin{pmatrix}4&2m\\2m&4\end{pmatrix}
\]
are even of signature $(1,1)$, hence occur as N\'eron--Severi lattices of projective complex K3 surfaces.
Choosing the connected component of the positive cone containing a prescribed class
(e.g. $H\in N_d$ with $H^2=2$, or $e_1\in\Lambda_m$ with $e_1^2=4$) realizes that class as ample.

\emph{Existence in positive characteristic.}
Let $L$ be one of the above rank-$2$ lattices, and let $d_L:=|\mathrm{disc}(L)|$. Bragg--Brakkee--V\'arilly-Alvarado construct a moduli stack of lattice-polarized K3 surfaces over $\Spec\Z$
and show that, after inverting the primes dividing $d_L$, the primitive locus is smooth away from a closed substack supported on supersingular locus and has the expected relative dimension \(20-\rk(L)\); see \cite[Thm.~1.1]{BBVA25} (also see \cite[\S1, \S2]{Riz06}).
Since the generic fiber (over $\Q$) is nonempty by the characteristic $0$ paragraph above, the image of this moduli stack in $\Spec\Z[1/d_L]$ contains a nonempty open subset.  For every prime \(p\) in that open subset, the fiber is nonempty,
which yields a lattice-polarized $K3$ surface over some finite type field extension of \(\F_p\). To arrange $\NS(X)\simeq L$, choose an irreducible component whose generic point is not contained in any Noether--Lefschetz locus corresponding to additional algebraic classes. Then the geometric generic fiber of that component has Neron--Severi group exactly \(L\).

\section{Irrational $F$-jumping numbers}\label{sec:irrational-f-jumps}

We now prove \textbf{Theorem 2}.

\subsection{The threefold cone and notation}

Let $(W,A)$ be any smooth projective surface from Sections~\ref{sec:examples-abelian-EE}--\ref{sec:examples-k3}
with irrational pseudo-effective threshold \footnote{Replacing $A$ by a sufficiently large multiple (equivalently, passing to a Veronese subring of $R(W,A)$), we may and do assume $0<t_0<1$. This anyway holds in all our explicit computations in Sections 4-7.}
\[
1>t_0 := t_0(W,A)\notin\Q,
\qquad
\kappa := 1-t_0\in(0,1).
\]

Form the cone $C=\Spec R(W,A)$ and let $\m\subset R(W,A)$ be the vertex ideal.
Let $R:=R(W,A)_\m$ be the local ring at the vertex.
Then $\dim R=3$ and $R$ is normal.
Note,$R$ is not $\Q$-Gorenstein whenever $t_0(W,A)$ is irrational, by Corollary \ref{cor:cone-not-QG}.

\begin{lemma}(\cite{Hara-Takagi,BSTZ10,ST})\label{lem:skoda-big}
Let $r$ be the minimal number of generators of $\m$ as an ideal of $R$. Then for all real $t\ge r$ one has
\[
\taub(R;\m^t) = \m\cdot \taub(R;\m^{t-1}).
\]
Consequently, for $t>r$, $t$ is an $F$-jumping number of $\taub(R;\m^t)$ if and only if $t-1$ is.
\end{lemma}

Fix a small real number $\varepsilon$ with
\[
0<\varepsilon<\kappa.
\]
We will compare $\taub(R;\m^{n+\kappa})$ and $\taub(R;\m^{n+\kappa-\varepsilon})$ for large integers $n$.

The upper bound for \emph{all} boundaries comes from Lemma~\ref{lem:k-ineq}:

\begin{lemma}\label{lem:upper-bound-all}
Let $\Gamma$ be any effective $\Q$-divisor on $C$ such that $K_C+\Gamma$ is $\Q$-Cartier.
Write $\pi^*(K_C+\Gamma)=K_Y+\Gamma_Y+kW_0$ as in Lemma~\ref{lem:k-ineq}.
Then for every real $t\ge 0$,
\[
\cJ(C,\Gamma;\m^t)\subseteq \m^{\floor{k+t}}
\subseteq \m^{\floor{t_0+1+t}}.
\]
In particular, for every integer $n\ge 0$,
\[
\cJ(C,\Gamma;\m^{n+\kappa}) \subseteq \m^{n+2},
\qquad
\cJ(C,\Gamma;\m^{n+\kappa-\varepsilon}) \subseteq \m^{n+1}.
\]
\end{lemma}

\begin{proof}
As in Proposition~\ref{prop:J-Gamma-s},
\[
E_\pi(t)=\ceil{K_Y-\pi^*(K_C+\Gamma)-tW_0}
=\ceil{-\Gamma_Y-(k+t)W_0}\le \ceil{-(k+t)W_0}=-\floor{k+t}W_0,
\]
Pushing forward gives $\cJ(C,\Gamma;\m^t)\subseteq \pi_*\cO_Y(-\floor{k+t}W_0)=\m^{\floor{k+t}}$.
By Lemma~\ref{lem:k-ineq}, $k\ge 1+t_0$, hence $\floor{k+t}\ge \floor{t_0+1+t}$.

Now set $t=n+\kappa$.
Since $t_0+1+\kappa=2$, we have $\floor{t_0+1+n+\kappa}=n+2$.
Similarly, for $t=n+\kappa-\varepsilon$ we have $t_0+1+\kappa-\varepsilon=2-\varepsilon$, so $\floor{t_0+1+n+\kappa -\varepsilon}=n+1$.
\end{proof}

The lower bound uses the special boundaries $\Gamma_s$ with $s>t_0$ and $s-t_0$ arbitrarily small.

\begin{lemma}\label{lem:lower-bound}
Fix a real $0<\delta<\kappa$. Then there exists a rational $s$ with
\[
t_0 < s < t_0+\delta
\]
such that $K_C+\Gamma_s$ is $\Q$-Cartier of index not divisible by $p$.
For such $s$, there exists $N_s\ge 0$ such that for all integers $n\ge N_s$,
\[
\tau(C,\Gamma_s;\m^{n+\kappa})=\m^{n+2}
\quad\text{and}\quad
\tau(C,\Gamma_s;\m^{n+\kappa-\delta})=\m^{n+1}.
\]
\end{lemma}

\begin{proof}
Choose $s\in\Q_{>0}$ with $t_0<s<t_0+\delta$ and with denominator not divisible by $p$.
Since $\delta<\kappa=1-t_0$, we also have $s<t_0+\delta<1$,
and then choose $\Gamma_s$ as in Proposition~\ref{prop:Gamma-s}.
By choosing $m$ in that construction
prime to $p$ (and taking $s$ with denominator prime to $p$), we may ensure that the Cartier index of
$K_C+\Gamma_s$ is not divisible by $p$.

By Proposition~\ref{prop:tau-Gamma-s}, there exists an integer $N_s$ such that for all real $t$ with $\lfloor t\rfloor\ge N_s$,
\[
\tau(C,\Gamma_s;\m^t)=\m^{\floor{s+1+t}}.
\]
Fix $n\ge N_s$.  Since $\kappa=1-t_0$, we compute
\[
  \lfloor s+1+n+\kappa\rfloor
  =\big\lfloor n+2+(s-t_0)\big\rfloor
  =n+2,
\]
because $0<s-t_0<\delta <1$.
Similarly,
\[
  \lfloor s+1+n+\kappa-\delta\rfloor
  =\big\lfloor n+2+(s-t_0)-\delta\big\rfloor
  =n+1,
\]
because $0<s-t_0<\delta$ implies $2+(s-t_0)-\delta\in(1,2)$.
This gives the desired equalities.
\end{proof}

Now we combine the upper and lower bounds with Schwede's sum formula \ref{def:big-test}.

\begin{proposition}\label{prop:big-values}
Fix $\varepsilon\in(0,\kappa)$.
Then there exists $N\ge 0$ such that for all integers $n\ge N$,
\[
\taub(R;\m^{n+\kappa}) = \m^{n+2}
\qquad\text{and}\qquad
\taub(R;\m^{n+\kappa-\varepsilon}) = \m^{n+1}.
\]
\end{proposition}

\begin{proof}
We work on $C$; localizing at the vertex does not change anything.

For the upper bound at $t=n+\kappa$, apply Lemma~\ref{lem:tau-in-J} to each boundary $\Gamma$ appearing in the sum formula:
\[
\tau(C,\Gamma;\m^{n+\kappa}) \subseteq \cJ(C,\Gamma;\m^{n+\kappa}) \subseteq \m^{n+2},
\]
where the final containment is Lemma~\ref{lem:upper-bound-all}.
Summing over $\Gamma$ gives $\taub(C;\m^{n+\kappa})\subseteq \m^{n+2}$.

For the matching lower bound, choose $\delta=\varepsilon$ and pick $s$ as in Lemma~\ref{lem:lower-bound}.
Then for $n \ge N_s$,
\[
\tau(C,\Gamma_s;\m^{n+\kappa})=\m^{n+2}.
\]
Since this summand appears in \ref{def:big-test}, we get $\m^{n+2}\subseteq \taub(C;\m^{n+\kappa})$.
Hence equality holds.

The argument for $t=n+\kappa-\varepsilon$ is identical:
Lemma~\ref{lem:upper-bound-all} bounds every summand by $\m^{n+1}$,
while Lemma~\ref{lem:lower-bound} provides a summand equal to $\m^{n+1}$.
\end{proof}

\begin{corollary}\label{lem:rigidity-window}
Fix $\epsilon\in(0,\kappa)$. Let $N=N(\epsilon)$ be as in Proposition~\ref{prop:big-values}.
Then for every integer $n\ge N+1$ the following hold:
\begin{enumerate}
\item $\tau_b(R;\m^t)=\m^{n+1}$ for all $t\in[\,n,\,n+\kappa-\epsilon\,]$.
\item $\tau_b(R;\m^t)=\m^{n+2}$ for all $t\in[\,n+\kappa,\,n+1\,]$.
\item Consequently, every $F$-jumping number of $\tau_b(R;\m^t)$ lying in the unit
interval $(n,n+1)$ must lie in the small interval $(\,n+\kappa-\epsilon,\;n+\kappa\,]$.
\end{enumerate}
\end{corollary}

\begin{proof}
    Immediate from Lemma \ref{lem:upper-bound-all}, Proposition \ref{prop:big-values} and Proposition \ref{prop:tau-Gamma-s}.
\end{proof}

To deduce that $n+\kappa$ itself is a jumping number, we need a no-accumulation result.

\begin{theorem}({\cite{S-Graf}})\label{thm:no-accum}
Let $(R,\m)$ be an $F$-finite normal local domain with an isolated singularity.
Then the set of $F$-jumping numbers of $\taub(R;\m^t)$ has no accumulation points in $\R$.
\end{theorem}

We now finish the main theorem.

\begin{proof}[Proof of \textbf{Theorem 2}]
Fix $\varepsilon\in(0,\kappa)$.
By Corollary~\ref{lem:rigidity-window}, for all $n\ge N(\epsilon)+1$ there exists at least one $F$-jumping number $\lambda_n$ in the interval $(n+\kappa-\varepsilon,n+\kappa]$.

Apply Lemma~\ref{lem:skoda-big} repeatedly to translate each $\lambda_n$ down by the integer $n-r$:
\[
\mu_n := \lambda_n-(n-r) \in (r+\kappa-\varepsilon,r+\kappa].
\]
Thus for every $\varepsilon>0$, there exist $F$-jumping numbers in $(r+\kappa-\varepsilon,r+\kappa]$.
In other words, $r+\kappa$ is a limit point of $F$-jumping numbers.

By Theorem~\ref{thm:no-accum}, the set of jumping numbers has no accumulation points, so the only way $r+\kappa$ can be a limit point
is if $r+\kappa$ itself is a jumping number.
Finally, by Lemma~\ref{lem:skoda-big},
if $r+\kappa$ is a jumping number then so is $n+\kappa$ for every $n\ge r$.
Since $\kappa\notin\Q$, these are irrational.
\end{proof}

\begin{corollary}\label{rem:all-but-finite}
    \textbf{Theorem 2} together with Theorem \ref{thm:no-accum} (=\cite[Theorem~A]{S-Graf}) shows that we have missed only finitely many $F$-jumps. The set of rational $F$-jumping numbers of $\tau_b(R;\mathfrak m^t)$ is finite.
\end{corollary}
\begin{proof}
We use Corollary~\ref{lem:rigidity-window}(3): Suppose, for the sake of contradiction, that there are infinitely many $F$-jumping numbers \emph{not}
of the form $n+\kappa$ for $n\in\Z_{\ge 0}$.  Then we can choose a sequence
$\varepsilon_j:=1/j\to 0$ and integers $n_j\ge \max\{r,\,N(\varepsilon_j)\}$ such that
there exists an $F$-jumping number
\[
\lambda_j \in (n_j,n_j+1]\setminus\{n_j+\kappa\}.
\]
In fact
$\lambda_j \in (n_j+\kappa-\varepsilon_j,\,n_j+\kappa)$.

Now apply Lemma~\ref{lem:skoda-big} repeatedly to translate by integers:
\[
\mu_j := \lambda_j-(n_j-r)\in (r+\kappa-\varepsilon_j,\,r+\kappa).
\]
Each $\mu_j$ is an $F$-jumping number of $\tau_b(R;\mathfrak{m}^t)$, and the intervals
$(r+\kappa-\varepsilon_j,\,r+\kappa)$ shrink to the point $r+\kappa$.
In particular, the set $\{\mu_j\}_{j\ge 1}$ has $r+\kappa$ as an accumulation point.
This contradicts Theorem \ref{thm:no-accum} (=\cite[Theorem~A]{S-Graf}). So we are done.

\end{proof}

\begin{corollary}\label{cor:asymptotic-floor}
In the setup of \textbf{Theorem~2}, there exists $T\gg 0$ such that for every real number $t\ge T$,
\[
\tau_b(R;\m^t)=\m^{\lfloor t+t_0(W,A)+1\rfloor}.
\]
Equivalently, for every integer $n\gg 0$ one has, for $\kappa=1-t_0(W,A)$,
\[
\tau_b(R;\m^t)=
\begin{cases}
\m^{n+1} & \text{if } n\le t < n+\kappa,\\[2pt]
\m^{n+2} & \text{if } n+\kappa \le t < n+1.
\end{cases}
\]
\end{corollary}
\begin{proof}
    Follows directly from Corollary \ref{rem:all-but-finite}, Corollary \ref{lem:rigidity-window} (also see Proposition \ref{prop:raywise-agreement}).
\end{proof}

\begin{remark}
    Let $\lambda_n$ denote the unique $F$-jumping number of $\tau_b(R;\m^t)$ in $(n,n+1)$ for all $n\gg 0$
whose existence and uniqueness are guaranteed by Lemma \ref{rem:all-but-finite}.
Then
\[
\kappa=\lim_{n\to\infty}(\lambda_n-n)
\qquad\text{and hence}\qquad
t_0(W,A)=1-\lim_{n\to\infty}(\lambda_n-n).
\]
In particular, in these cone setup the numerical invariant $t_0(W,A)$ is determined by the tail of the
$F$-jumping numbers of $\tau_b(R;\m^t)$.
\end{remark}

\begin{remark}
    In the cone setup of \textbf{Theorem~2} (so that $\kappa\notin\mathbb{Q}$), there does not exist a real number $T>0$ and
a \emph{finite} collection of boundaries $\Delta_1,\dots,\Delta_M$ on $X=\Spec R$ such that
$K_X+\Delta_i$ is $\mathbb{Q}$-Cartier of index not divisible by $p$ and
\[
\tau_b(R;\mathfrak m^t)\;=\;\sum_{i=1}^M \tau\bigl(X;\Delta_i,\mathfrak m^t\bigr)
\qquad\text{for every }t\ge T.
\]
In particular, there is no \emph{single} boundary $\Delta$ computing $\tau_b(R;\mathfrak m^t)$ for all $t\gg 0$.
\end{remark}

\begin{corollary}\label{cor:anticanonical-not-fg}
In the cone setup of \textbf{Theorem~2}, the anti-canonical algebra $\mathcal{R}(-K_R) $
is not finitely generated.
\end{corollary}

\begin{proof}
Follows from \cite[Theorem~B]{Fin-gen}. 
\end{proof}

\begin{remark}
We caution the reader that the implication 
\[
 t\ \text{is an $F$-jumping number of }\tau_b(R;\mathfrak m^t)
\ \Longrightarrow\
p^e t\ \text{is an $F$-jumping number of }\tau_b(R;\mathfrak m^t)
\]
(\cite[Lemma~3.25]{BSTZ10}) does \emph{not} hold in non-$\Q$-Gorenstein world. Assume, for the sake of contradiction, that such a fixed $e$ exists.
By \textbf{Theorem 2.}, $r+\kappa$ is an
$F$-jumping number of $\tau_b(R;\mathfrak m^t)$.
Then, $p^{eN}(r+\kappa)$ is an $F$-jumping number for every $N\ge 1$. Lemma~\ref{lem:skoda-big} applied repeatedly shows that $r+\{p^{eN}\kappa\}$ are $F$-jumping numbers for all $N\ge 1$.
Since $\kappa$ is irrational, $\{p^{eN}\kappa\}$ are all distinct, hence
$\{r+\{p^{eN}\kappa\}\}_{N\ge 1}$ is an infinite subset of the compact interval $[r,r+1]$ and therefore has an accumulation
point in $\mathbb R$. This contradicts Theorem \ref{thm:no-accum} (=\cite[Theorem~A]{S-Graf}).

\end{remark}

The floor formulas in Propositions \ref{prop:tau-Gamma-s} and \ref{prop:J-Gamma-s} show that the irrational jump $\kappa$ is a limit of rational jumping numbers for auxiliary $\mathbf{Q}$-Gorenstein boundaries.

\begin{cor}\label{cor:rationalapprox}
Fix a sequence of rational numbers $s_i\in \mathbf Q$ with $s_i>t_0$ and $s_i\downarrow t_0$.
Set $t_i:=1-s_i$.
Then:
\begin{enumerate}
\item[(1)] each $t_i$ is a jumping number for the family $t\mapsto \tau(R;\Gamma_{s_i},\mathfrak m^t)$ in characteristic $p>2$ (and for the corresponding multiplier ideal family $t\mapsto \mathcal J(C,\Gamma_{s_i};\mathfrak m^t)$ in characteristic $0$);
\item[(2)] $t_i\in \mathbf Q$ for all $i$ and $t_i\uparrow \kappa=1-t_0$.
\end{enumerate}
\end{cor}

\begin{proof}
Immediate from Proposition \ref{prop:tau-Gamma-s} (Proposition \ref{prop:J-Gamma-s}).
\end{proof}

\begin{remark}
Corollary~\ref{cor:rationalapprox} does \emph{not} say that $\kappa$ is a limit of rational jumping numbers for a \emph{fixed} log-$\mathbf{Q}$-Gorenstein triple. Rather, one varies the boundary $\Gamma_s$.
\end{remark}

\subsection{Future directions}
\begin{que}
It is quite frustrating to me that we do not know if the \emph{first $F$-jumping numbers} in any of these examples are irrational or not! More surprisingly, it seems we do not know in general whether, for an $F$-jumping number $\lambda < r$, $\lambda +1$ is also an $F$-jumping number or not! 
\end{que}
However, the corresponding assertion is true for multiplier ideals in char 0, see \cite[Remark 1.14]{ELSV} and \cite[Corollary 1.3]{Brad}. One might hope to apply techniques of \cite{Ajit-Simper-1} to get \cite[Application 7.2]{Ajit-Simper-1} without assuming $\Q$-Gorensteinness.

\begin{assume}(cf.\cite[App 7.2]{Ajit-Simper-1}, \cite[Rem 1.14]{ELSV} \cite[Cor 1.3]{Brad})\label{ass:translate}
If $\lambda < r$\footnote{where $r$ is the number of generators of $\m$.} is an $F$-jumping number of $\tau_b(R;\m^t)$, then
$\lambda+N$ is an $F$-jumping number for infinitely many integers $N\ge 0$.
\end{assume}

\begin{theorem}\label{thm:first-jump-irrational}
In the cone setup of \textbf{Theorem 2} (so $\kappa=1-t_0(W,A)\in(0,1)$ is irrational),
assume Assumption \ref{ass:translate}.
Then every $F$-jumping number $\lambda$ of $\tau_b(R;\m^t)$ satisfies
$\{\lambda\}=\kappa$.
In particular, the first $F$-jumping number is irrational.
\end{theorem}

\begin{proof}
We provide atleast 3 proofs:
\begin{itemize}
    \item Fix an arbitrary $\epsilon\in(0,\kappa)$ and let $N(\epsilon)$ be as in
Lemma~\ref{lem:rigidity-window}.
Let $\lambda$ be any $F$-jumping number of $\tau_b(R;\m^t)$.
By Assumption \ref{ass:translate}, there exist infinitely many integers $M\ge 0$
such that $\lambda+M$ is also an $F$-jumping number. Choose such an $M$ so large that $n:=\lfloor \lambda+M\rfloor \ge N(\epsilon)+1$.
Then by
Corollary ~\ref{lem:rigidity-window}(3) we must have
\[
\lambda+M\in(n+\kappa-\epsilon,\;n+\kappa].
\]
Subtracting $n$ shows $\{\lambda\}=\{\lambda+M\}\in(\kappa-\epsilon,\kappa]$.
Since $\epsilon\in(0,\kappa)$ was arbitrary, we conclude $\{\lambda\}=\kappa$.
Because $\kappa$ is irrational, so is $\lambda$.

\item Let $\lambda_1$ be the first $F$-jumping number of $\tau_b(R;\mathfrak m^t)$.
Suppose for contradiction that $\lambda_1\in\Q$.
By Assumption \ref{ass:translate}, there exist infinitely many integers $N\ge 0$ such that $\lambda_1+N$
is an $F$-jumping number of $\tau_b(R;\mathfrak m^t)$. But $\lambda_1+N\in\Q$ for all $N\in\Z$, hence $\tau_b(R;\mathfrak m^t)$ has infinitely many rational $F$-jumping numbers. This contradicts Remark \ref{rem:all-but-finite}. Therefore $\lambda_1\notin\Q$.
\item Let $\lambda$ be any $F$-jumping number for $\tau_b(R;\m^t)$.
By Assumption \ref{ass:translate} there are infinitely many integers $N\ge 0$ such that $T:=\lambda+N$ is also an $F$-jumping number. Fix such an $N$ so large that Proposition~\ref{prop:raywise-agreement} applies
to $t=\lambda$, i.e.\ so that
\[
\tau_b(R;\m^{\lambda+N})=\cJ_{dFH}(R;\m^{\lambda+N}).
\]
We always have $\tau_b(R;\m^t)\subseteq \cJ_{dFH}(R;\m^t)$ for all real $t\ge 0$
(by Lemma~\ref{lem:tau-in-J} summed over boundaries).
Therefore, if $\cJ_{dFH}(R;\m^t)$ were locally constant at $t=T$, then for all sufficiently
small $\varepsilon>0$ we would have
\[
\tau_b(R;\m^{T-\varepsilon})
\subseteq \cJ_{dFH}(R;\m^{T-\varepsilon})
=
\cJ_{dFH}(R;\m^{T})
=
\tau_b(R;\m^{T}),
\]
and since $\tau_b(R;\m^{T-\varepsilon})\supseteq \tau_b(R;\m^{T})$ by monotonicity,
this would force $\tau_b(R;\m^{T-\varepsilon})=\tau_b(R;\m^{T})$, contradicting that $T$ is an $F$-jumping number.  Hence $T$ must be a jumping number of $\cJ_{dFH}(R;\m^t)$ as well. Finally, by Theorem \ref{thm:Jmax-cone}, $T=\lambda+n\in \kappa+\Z$, and therefore $\lambda\in \kappa+\Z$ showing $\lambda$ is irrational.
\end{itemize}

\end{proof}

\begin{que}
    At present we do not know an explicit example in the literature of a \emph{transcendental}
pseudo-effective threshold (or a transcendental jumping number). On the other hand, there are closely related\footnote{$\vartheta(W,A)=\inf\{t\in\R\mid \vol_W(tA-K_W)>0\}$} asymptotic invariants that \emph{are} known to be transcendental in explicit examples, notably volume functions of divisors
and linear series; see \cite{KLM-vol} and \cite{pi-vol}, who in particular show that $\pi$ can occur as a volume. 
\end{que}

However, it is quite easy to construct transcendental jumping numbers for asymptotic multiplier ideals of a graded sequence rather than for the multiplier ideals of a single ideal: Let $R=k[x,y]$. For each $m\ge 1$, define the monomial ideal
\[
\mathfrak a_m:=\bigl(x^iy^j\mid i,j\in \Z_{\ge 0},\ i+\pi j\ge m\bigr)\subset R.
\]
Then $\mathfrak a_\bullet=\{\mathfrak a_m\}_{m\ge 1}$ is a graded sequence of ideals. By Howald's theorem \cite{Howald}, its asymptotic multiplier ideals are
\[
\mathcal J(\mathfrak a_\bullet^t)=\bigl(x^ay^b\mid a,b\in \Z_{\ge 0},\ a+1+\pi(b+1)>t\bigr),
\]
and the set of jumping numbers is exactly
\[
\{a+1+\pi(b+1)\mid a,b\in \Z_{\ge 0}\}.
\]
In particular $1+\pi$ is a transcendental jumping number.

\bibliographystyle{amsalpha}

\bibliography{ref}

\end{document}